\documentclass[10pt]{article}%
\usepackage{amsfonts}
\usepackage{amsmath}
\usepackage{amssymb,amsthm,bm}
\usepackage{graphicx}
\usepackage{mathtools}
\usepackage{color}
\usepackage{url}
\providecommand{\U}[1]{\protect\rule{.1in}{.1in}}
\newtheorem{theorem}{Theorem}

\newtheorem{corollary}[theorem]{Corollary}

\newtheorem{lemma}[theorem]{Lemma}

\newtheorem{proposition}[theorem]{Proposition}
\newtheorem{remark}[theorem]{Remark}

\newcommand{\R}{\mathbb{R}}
\newcommand{\Z}{\mathbb{Z}}
\newcommand{\N}{\mathbb{N}}
\newcommand{\C}{\mathbb{C}}
\newcommand{\bydef}{\,\stackrel{\mbox{\tiny\textnormal{\raisebox{0ex}[0ex][0ex]{def}}}}{=}\,}
\newcommand{\bx}{{\bar{x}}}
\newcommand{\bu}{{\bar{u}}}
\newcommand{\bv}{{\bar{v}}}
\newcommand{\bw}{{\bar{w}}}
\newcommand{\tx}{{\tilde{x}}}
\newcommand{\tu}{{\tilde{u}}}

\newcommand{\dagA}{A^\dagger}
\newcommand{\real}{{\textbf{real}}}
\newcommand{\inc}{\bm{\iota}^{(m)}}

\begin{document}

\date{}

\title{Torus knot choreographies in the $n$-body problem}

\author{
Renato Calleja\thanks{IIMAS, %Instituto de Investigaciones en Matem\'aticas Aplicadas y en Sistemas,
Universidad Nacional Aut\'onoma de M\'exico, Apdo. Postal 20-726, C.P. 01000, México D.F., México. {\tt calleja@mym.iimas.unam.mx}}
\and
Carlos Garc\'{\i}a-Azpeitia\thanks{IIMAS, Universidad Nacional Aut\'onoma de M\'exico, Apdo. Postal 20-726, C.P. 01000, México D.F., México. {\tt cgazpe@ciencias.unam.mx}}
\and
Jean-Philippe Lessard\thanks{McGill University, Department of Mathematics and Statistics, 805 Sherbrooke Street West, Montreal, QC, H3A 0B9, Canada. {\tt jp.lessard@mcgill.ca}}
\and
J.D. Mireles James \thanks{Florida Atlantic University, Department of Mathematical Sciences, 777 Glades Road, Boca Raton, FL 33431, USA. {\tt jmirelesjames@fau.edu}}
}

\maketitle

\begin{abstract}
We develop a systematic approach for proving the existence of
choreographic solutions in the gravitational $n$ body problem.
Our main focus is on spatial torus knots: that is, periodic motions where 
the positions of all $n$ bodies follow a single closed which winds around a $2$-torus
in $\mathbb{R}^3$. 
After changing to rotating coordinates and exploiting symmetries, the equation of a
choreographic configuration is reduced to a delay differential equation (DDE)
describing the position and velocity of a single body. We study periodic
solutions of this DDE in a Banach space of rapidly decaying Fourier
coefficients. Imposing appropriate constraint equations lets us isolate
choreographies having prescribed symmetries and topological properties. Our
argument is constructive and makes extensive use of the digital computer. We
provide all the necessary analytic estimates as well as a working implementation
for any number of bodies. We illustrate the utility of the approach by
proving the existence of some spatial choreographies for $n=4,5,7$,
and $9$ bodies.
\end{abstract}

\section{Introduction}

\label{sec:intro}

A \emph{choreography} is a periodic solution of the gravitational $n$-body
problem, where $n$ equal masses follow the same path. Circular
choreographies with masses located at the vertices of
a regular $n$-gon were already studied by
Lagrange in the Eighteenth Century.  The first choreography differing from 
a polygon was discovered by
Moore in \cite{Mo93} and has three bodies moving around the now famous
figure-eight. Chenciner and Montgomery in \cite{ChMo00} gave a rigorous
mathematical proof of the existence of this figure eight orbit by minimizing
the action for Newton's equation. The name \textit{choreographies} was adopted
after the work of Sim\'o \cite{Si00} on numerical computation
of choreographic solutions.

The variational approach to the existence of choreographies consists of
finding critical points of the classical Newtonian action subject to
appropriate symmetry constraints. The main obstacle to this approach is the
existence of paths with collisions. Terracini and Ferrario in \cite{FeTe04}
gave conditions on the symmetries which imply that a minimizer is free of
collisions (this is called the rotating circle property). Although a lot of
simple choreographies have been found numerically since Sim\'o \cite{Si00},
rigorous proofs using only analytical methods are difficult. Notable
exceptions include works on: the figure-eight of three bodies \cite{ChMo00},
the rotating $n$-gon \cite{BT04}, the figure-eight type for odd bodies
\cite{FeTe04} and the super-eight of four bodies \cite{Sh14}. Other
variational approaches related to existence of planar choreographies can be
found in \cite{BaTe04,Ch03,Fe06,FePo08,TeVe07,Wa16} and the references therein.

The difficulties just mentioned have led some authors
to develop mathematically rigorous computer assisted proofs (CAPs) for
choreographies. This is a natural alternative to pen-and-paper
analysis since both the discovery and many subsequent studies of
choreographies employ numerical methods. The interested reader will want to
consult for example the works of Kapela, Sim\'{o}, and Zgliczy\'nski
\cite{KaSi07,KaSi17,KaZg03} for both CAPs of existence for planar
choreographies and mathematically rigorous stability analysis. See also Remark
\ref{rem:comparisons} below.

Recall now that a $(p, q)$-\textit{torus knot} is an embedding of $\mathbb{S}^1$ into 
a two torus $\mathbb{T}^2 \subset \mathbb{R}^3$, winding 
$p$ times around one generating circle of the torus and $q$ times around 
the other, with $p$ and $q$ coprime and neither equal to zero.  
The embedding of the two torus 
is required to be unknotted in $\mathbb{R}^3$.  A torus knot may or 
may not be a trivial when viewed as a knot in $\mathbb{R}^3$.
Indeed, it is trivial if and only if either $p$ or $q$ is equal to $\pm 1$.
The idea is illustrated in Figure \ref{fig:knots}.

\begin{figure}[t!]
\centering
\includegraphics[width = 5in]{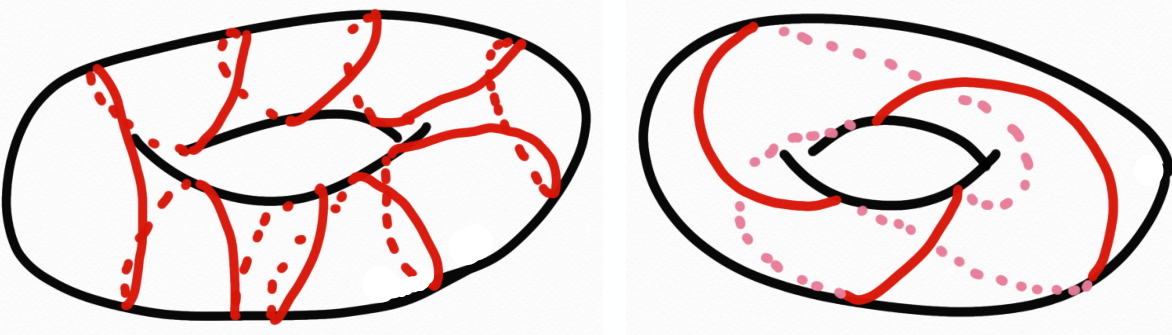} \caption{
\textbf{Spatial torus knots}: given an unknotted 
$2$-torus embedded in $\mathbb{R}^3$, a $(p,q)$-torus knot is a 
non-contractible curve embedded into the surface of the torus.
The curve then winds $p$ and $q$ times respectively around the 
generating circles of the torus (with $p$ and $q$ co-prime). 
It is a basic result that a $(p,q)$ torus knot is trivial 
as a knot in $\mathbb{R}^n$ if and 
only if either $p$ or $q$ is $\pm 1$.
The left frame illustrates a torus knot which is a trivial knot 
in $\mathbb{R}^3$, while the right frame illustrates a 
non-trivial $(3,2)$-knot -- in fact a trefoil.}
\label{fig:knots}%
\end{figure}

A difficult problem in this area is to prove the existence 
of spatial torus knot choreographies. Indeed when both topological and
symmetric constraints are involved, it is difficult to prove the coercitivity of
the action. For this reason few results with topological constraints are
available. A notable exception is a torus knot
choreography for $3$-bodies obtained by Arioli, Barutello, and Terracini
in \cite{MR2259202}, where the authors localize
a mountain pass solution of the Newtonian action in a rotating frame.
Again the result is obtained by means of CAP, not variational methods.
In general it is hard to
determine whether a critical point of the action is a spatial torus-knot choreography.
We provide a systematic procedure to obtain countable families of
torus knots for any number of bodies. 

\textbf{Contribution:} 
%A torus knot is a special kind of knot that lies on the surface of an  torus. 
%A {\em$(p,q)$}-torus knot can be specified by the coprime integers $p$ and $q$ 
%that represent the winding numbers in two interior circles.  
\emph{The main result of the present work is 
%the first mathematically rigorous proof 
to give mathematically rigorous existence proofs for  {\em$(p,q)$}-torus knot
choreographies in the {\em$n$}-body problem for several different 
values of $n$.} 

Our approach is functional analytic 
(a choreography is a zero of a nonlinear operator posed on a Banach space) 
and computer-assisted.  When it succeeds it produces countably many verified results.
For example we establish the existence of the $5$-body trefoil knot choreography illustrated 
in Figure \ref{fig:trefoil}, 
and the existence of countable many choreographies close to it. 
We describe the pen and paper estimates for any number of bodies and, while we 
illustrate the method for only few 
explicit examples, our setup and resulting implementation apply (in principle) to 
any spatial choreography.

\begin{figure}[t!]
\centering
\includegraphics[width = 4.5in]{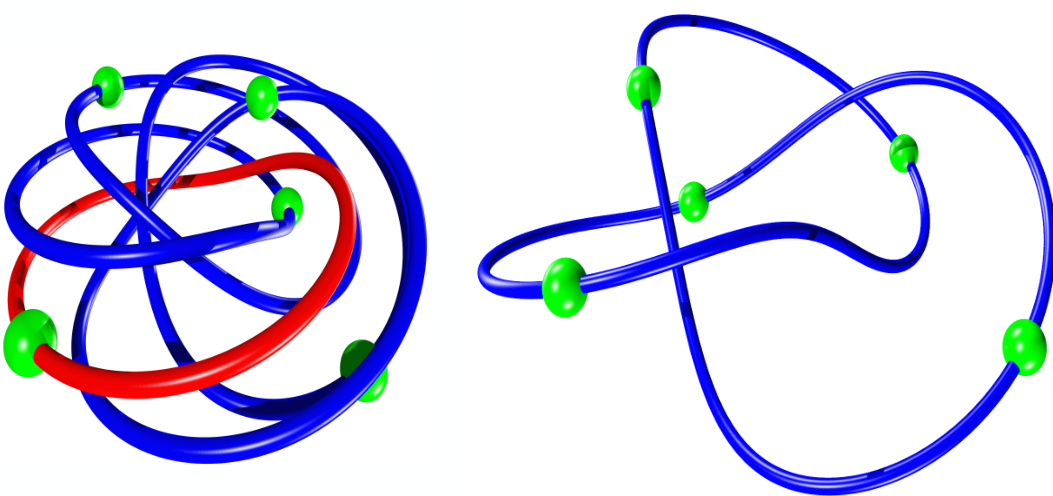} \caption{Example of a spatial
trefoil choreography for 5 bodies: Left frame (rotating coordinates) the red
loop illustrates the periodic orbit of the delay differential equation whose
existence we prove using the methods of the present work. The four remaining
loops are obtained by symmetry, giving a periodic orbit of the full 5 body
problem in rotating coordinates. Right frame (inertial coordinates) the 5 body
orbit converted to rotating coordinates. The result is a spatial torus knot
with the topology of a trefoil. }%
\label{fig:trefoil}%
\end{figure}

Before describing our approach in detail we recall several related developments.  
In \cite{ChFe08} it is observed that choreographies
appear in dense sets along the vertical Lyapunov families attached to the
relative equilibrium solutions given by the planar $n$-gon . Existence of vertical Lyapunov families
follows from the Weinstein-Moser theory and, when the frequency varies
continuously, the authors obtain the existence of an infinite number of
choreographies along these vertical families. This hypothesis however
has been verified only for some families with $n=3,4,5,6$ and even though
similar computations can be carried out for other values of $n$, it is an open
problem to establish the hypothesis for all $n$.

The existence of global Lyapunov families arising from the polygonal relative
equilibrium of the rotating problem 
was established in \cite{GaIz11,GaIz13} for all $n$.  By saying that these 
families are global what we mean that, in the space of normalized $2\pi$ periodic
solutions, the families form a continuum set with at least one of the
following properties: either the Sobolev norm of the orbits in the family goes
to infinity, the period of the orbits goes to infinity, the family ends in an
orbit with collision, or the family returns to another equilibrium solution.
This fact is proved using $G$-equivariant degree theory \cite{IzVi03} where
$G=\mathbb{Z}_{n}\times\mathbb{Z}_{2}\times SO(2)\times S^{1}$ acts as
permutations, $z$-reflection and $(x,y)$-rotations of bodies, and time shift
respectively. In addition the analysis of \cite{GaIz11,GaIz13} concludes that
the Lyapunov families have the symmetries of a twisted subgroup of $G$.

Specifically, let $(w_{j},z_{j})\in\mathbb{C}\times\mathbb{R}$ represents the
planar and spatial coordinates of the $j$-th body in a rotating coordinate
frame with frequency $\sqrt{s_{1}}$, where
\begin{equation}
s_{1}=\frac{1}{4}\sum_{j=1}^{n-1}\frac{1}{\sin(j\zeta/2)},\qquad\zeta
=\frac{2\pi}{n}. \label{s1}%
\end{equation}
The $n$-polygon consisting of $n$ bodies on the unit circle $w_{j}=e^{ij\zeta}$ is an
equilibrium solution of Newton's equations in a co-rotating frame.
 After normalizing the period to
$2\pi$, the planar Lyapunov families arising from this equilibrium polygon
have the planar symmetries,
\begin{equation}
w_{j}(t)=e^{ij\zeta}w_{n}(t+jk\zeta), \label{SyP}%
\end{equation}
and the spatial symmetries%
\begin{equation}
z_{j}(t)=z_{n}(t+jk\zeta). \label{SyS}%
\end{equation}
For $1 \leq j \leq n-1$ the $j$-th body follows 
an identical path as the $n$-th body, after
a rotation in space and a shift in time. It is proved in \cite{GaIz13} that
taking $k=2,.., n-2$ in the planar case 
gives the $n-3$ planar Lyapunov families, and that taking
$k=1,...,n-1$ in the spatial case gives the $n-1$ vertical Lyapunov families.

We stress that the $G$-equivariant degree theory provides only an alternative
concerning the global behavior of the Lyapunov families. Without additional
information we do not know \textit{what actually happens along a given
branch}. This question is considered in \cite{CaDoGa18}, where the authors 
conduct a numerical exploration of the
global behavior of the Lyapunov families using the software package AUTO (e.g.
see \cite{MR635945}). 

Let $p, q \in \mathbb{Z}$ be relatively prime such that $kq-p\in n\mathbb{Z}$.
It is proved in \cite{CaDoGa18} that an orbit with the 
symmetries defined in equations \eqref{SyP} and
\eqref{SyS}, and frequency
\begin{equation}
\omega=\sqrt{s_{1}}\frac{p}{q}, \label{di}%
\end{equation}
is a \emph{simple choreography} when converted back to 
the inertial reference frame. In the case
that $p$ and $q$ do not satisfy this diophantine equation, the solution in the
inertial frame corresponds to a \emph{multiple choreographic solution}
\cite{Ch03}, while the case that $\omega/\sqrt{s_{1}}$ is irrational implies
that the solution is quasiperiodic. Since the set of rational numbers $p/q$
satisfying the diophantine relation (\ref{di}) is dense, one has the
following: when the frequency $\omega$  varies
continuously along the Lyapunov family, there are infinitely many orbits 
in the rotating frame that
correspond to simple choreographies in the inertial frame.

The authors of \cite{CaDoGa18} give compelling numerical evidence 
which suggests that an axial family of solutions appears after a symmetry-breaking
bifurcation from the vertical Lyapunov family in the rotating $n$-body problem. 
The numerics suggest that this axial family has the symmetries of equations 
\eqref{SyP} and \eqref{SyS}. It is shown further in the same reference
that, if the hypothesized axial family exists, then orbits
in this family correspond  to choreographies 
in the inertial frame which wind $p$ and $q$ times 
around the generators of a $2$-torus.
That is, the periodic orbits in this alleged axial family 
give rise to $(p,q)$-torus knot choreographies for the $n$-body problem. 

A more refined description of our contribution is that we prove the 
existence of this axial family.
Using the symmetries (\ref{SyP}, \ref{SyS}) in Newton's laws 
we reduced the equations of motion to a single equation 
describing the motion of the $n$-th body
$u_{n}=(w,z)\in\mathbb{C\times R}$.  
The equation is a delay differential equation (DDE) 
with multiple constant delays.  More explicitly, we have
\begin{equation}%
\begin{array}
[c]{rl}%
\hspace{-0.24cm}\ddot{w}(t)+2\sqrt{s_{1}}i\dot{w}(t) & \hspace{-0.25cm}%
=\displaystyle s_{1}w(t)-\sum_{j=1}^{n-1}\frac{w(t)-e^{ij\zeta}w(t+jk\zeta
)}{\left(  \left\vert w(t)-e^{ij\zeta}w(t+jk\zeta)\right\vert ^{2}+\left\vert
z(t)-z(t+jk\zeta)\right\vert ^{2}\right)  ^{3/2}}\\
\ddot{z}(t) & \hspace{-0.25cm}=\displaystyle-\sum_{j=1}^{n-1}\frac
{z(t)-z(t+jk\zeta)}{\left(  \left\vert w(t)-e^{ij\zeta}w(t+jk\zeta)\right\vert
^{2}+\left\vert z(t)-z(t+jk\zeta)\right\vert ^{2}\right)  ^{3/2}}.
\end{array}
\label{eq:reduced_equations}%
\end{equation}
For any number of bodies, these reduced equations \eqref{eq:reduced_equations}
represents a system of six scalar equations with multiple constant delays.

Our computer assisted arguments are in the functional analytic tradition of
Lanford, Eckamnn, Koch, and Wittwer
\cite{MR648529,MR759197,MR883539,MR727816}, and build heavily on the 
earlier work of \cite{Le10,MR2871794,Le18} on DDEs. 
More precisely, we formulate the existence proofs on a Banach
space of Fourier coefficient sequences. The delay operator acts as a
multiplicative (diagonal) operator in Fourier coefficient space, and the
regularity of periodic solutions translates into rapid decay of the Fourier
coefficients. Indeed, as was shown in \cite{MR0322315}, a periodic
solution of a delay differential equation with analytic nonlinearity is
analytic when the delays are constant. Then we know a-priori that the Fourier
coefficients of a periodic solution of Equations \eqref{eq:reduced_equations}
decay exponentially fast.

An important feature of Equations \eqref{eq:reduced_equations} is the
conservation of energy, which allows us to fix a desired frequency for the
periodic solution \emph{a-priori}. This reduction greatly simplifies the
analysis of the delay differential equation in Fourier space, but requires
adding an unfolding parameter to balance the system. In addition we utilize
\textit{automatic differentiation} as in
\cite{Mi18,jp_jaime_j,jp_chris_jb_DDE}, and reformulate
\eqref{eq:reduced_equations} as a problem with polynomial nonlinearities. The
polynomial problem is amenable to straight forward analysis exploiting the
Banach algebra properties of the solution space and we use the FFT
algorithm as in \cite{MR3833658}. The cost of this simplification is that each
additional body augments the system with a single additional scalar equation
and a single additional unfolding parameter. Finally we validate the existence
of solutions by means of a Newton-Kantorovich argument exploiting 
the \emph{radii polynomial approach} as in \cite{LeMi16}.

%(JAY): I thought this paragraph repeats almost exactly things already said above:
%This results in a systematic method for proving the existence of planar and
%spatial choreographies in the $n$-body problem, which scales well with the
%number of bodies, and which in principle applies for large values of $n$. The
%estimates are presented in general, and our implementation works for any
%number of bodies, provided we begin with good enough initial data in Fourier space.

We conclude this introduction by mentioning some interesting problems for future study. 
The zero finding problem studied in the present work is amenable to validated
continuation techniques as discussed in
\cite{Le18,MR2630003,MR2338393,MR3125637}. A follow up study will investigate
global properties of continuous families of spatial choreographies in the $n$ body problem, and
study bifurcations encountered along the branches. In this way we hope to
prove for example the conjecture of Marchal/Chenciner \cite{ChFe08} that the Lagrange
triangle is connected with the figure-eigth choreography trough Marchal's P-12
family \cite{Ma00}.  We also remark that all the choreographies shown to exist in the 
present work are unstable. Actually, the only known stable choreographies are close to the 
figure eight for $n=3$. Stability of torus knots in the $n=3$ is being investigated 
in a forthcoming paper.

Let us also mention that 
the procedure developed in this paper could be adapted to prove existence of 
asymmetric planar or spatial choreographies. These choreographies do appear in dense 
sets of symmetry-breaking families from planar and spatial Lyapunov families. 
Furthermore, this procedure could be adapted to study choreography solutions in problems with 
other potentials, such as $r^{-\alpha }$\ (with $\alpha =1$\ being the
gravitational case, $\alpha <1$\ the weak force case, and $\alpha >2$\ the
strong force case). It could also be adapted to
Hamiltonian systems with different radial potentials, as long as the
polynomial embedding (see Section~\ref{eq:automatic_differentiation}) can be done. 
An interesting problem would be to adapt the method to validate choreographies
in families that bifurcate from the polygonal equilibrium in DNLS equations
\cite{CaGaDo_2} or the $n$-vortex problem on the plane, disk, or sphere 
\cite{CaGaDo_1}.

\begin{remark}
[CAPs in celestial mechanics and dynamics of DDEs]\label{rem:CAP}
\emph{Numerical calculations have been central to the development of celestial
mechanics since the late Nineteenth and early Twentieth centuries. The reader
interested in historical developments before the age of the digital computer
can consult the works of George Darwin, Francis Ray Moulton, and the group in
Copenhagen led by Elis Str\"{o}mgren \cite{MR1554890,MR0094486,Stromgren}.
Problems in celestial navigation and orbit design helped drive the explosion
of scientific computing during the space race of the mid Twentieth century. A
fascinating account and a much more complete bibliography are found in the book
\cite{theoryOfOrbits}.}

\emph{As researchers developed computer assisted methods of proof for
computational dynamics it was natural to look for challenging open
problems in celestial mechanics. The relevant literature is rich and
we direct the interested reader to the works of \cite{MR2112702,
MR1947690,MR2259202, MR2174417, MR1961956} for a much more
complete view of the literature. Other authors have studied center manifolds
\cite{MR2902618}, transverse intersections of stable/unstable manifolds
\cite{Mi18,MR3032848}, Melnikov theory \cite{MR3567489},
Arnold diffusion and transport \cite{MR2785975,MR2824484,MR3038221}, and
existence/continuation/bifurcation of Halo orbits
\cite{jp_jaime_j,danielContinuation} -- all in gravitational $n$-body problems
and all using computer assisted arguments. Especially relevant to the present
work are the computer assisted existence and KAM stability proofs for $n$-body
choreographies in \cite{KaSi07,KaSi17,KaZg03,MR2259202}. (See also Remark
\ref{rem:comparisons} below). Again, the references given in the preceding
paragraph are meant only to point the reader in the direction of the relevant
literature. A more complete view of the literature is found in the references
of the cited works.}

\emph{The present work grows out of the existing
literature on CAPs for dynamics of DDEs, the foundations of which were laid in
\cite{Le10}. The work just cited studied periodic solutions -- as
well as branches of periodic solutions -- for scalar DDEs with a single delay and 
polynomial nonlinearities.
Extensions to multiple delays appear in \cite{MR2871794}, and more 
recent work  considers systems of DDEs with non-polynomial
nonlinearities \cite{jp_chris_jb_DDE}. The interested reader can consult
the works of \cite{Le18,jonopJP_Kon,minamoto,MR2177094} for more complete
discussion of this area. We mention also the recent Ph.D. thesis of Jonathan
Jaquette, who settled the decades old conjectures of Wright and
Jones about the global dynamics of Wright's equation
\cite{jono_jones,MR3779642} using ideas from this field. Another approach to
computer assisted proof for periodic orbits of DDEs -- based on rigorous
integration of the induced flow in function space -- is found in
\cite{piotrDDE}. }

\emph{In spite of the picture painted above, computer assisted methods of
proof are regularly applied outside the boundaries of celestial mechanics and
delay differential equations. For a broader perspective on the area, still
focusing on nonlinear dynamics, we refer to the review articles
\cite{notices_jb_jp,Jay_Kons_Review} and to the book of Tucker
\cite{MR2807595}.}
\end{remark}

\begin{remark}
[Phase space and functional analytic approaches]\label{rem:comparisons} \emph{
%Since computer assisted methods of proof have been used already by other
%authors studying $n$-body choreographies, it is appropriate to say few words
%about the novelty of the present work. 
The existence proofs for planar
choreographies in \cite{KaSi07} and \cite{KaZg03}, the proof of the spatial 
mountain pass solution in \cite{MR2259202}, and the proof of KAM
stability of the figure eight choreography in \cite{KaSi17}
use a different setup from that developed in the present work. 
More precisely,  the works just mentioned study directly the Newtonian equations of 
motion in phase space. The works of \cite{KaSi07,KaZg03,KaSi17}
exploit the powerful CAPD library for rigorous integration of ODEs to
construct mathematically rigorous arguments in appropriate Poincar\'e
sections. See \cite{capdReference,MR1930946} for more complete discussion and references to
the CAPD library. 
The work of  \cite{MR2259202} utilizes a functional analytic method akin to
that of the present work, but applied directly to periodic orbits for the
Hamiltonian vector field rather than reducing to the delay differential equation as
in the present work.  
}

\emph{In the case of the planar choreography problem the phase space is of
dimension $4n$, while the spatial choreography problem scales like $6n$. 
These figures are in some sense conservative, as applying the topological arguments of
\cite{KaSi07,KaZg03} require integration of the equations of first variation
(and equations of higher variation in the case of the KAM stability argument).
}

\emph{The setup of the present work considers six scalar
equations, independent of the number of bodies considered. This is
a dramatic reduction of the dimension of the problem. This dimension
reduction facilitates consideration of -- in principle -- choreographies
involving any number of bodies. A technical remark is that
our implementation uses automatic
differentiation to reduce to a polynomial nonlinearity, adding one additional
scalar equation for each body being considered. This brings our count to
$6+(n-1)$ scalar equations. While this quantity scales with $n$ much better than the
$6n$ mentioned above, we stress that our implementation could be improved
using techniques similar to those discussed in
\cite{MR2259202,MR3124898,jordiAlexKAM} for evaluation of non-polynomial
nonlinearities on Fourier data. With such an improvement our approach would 
consider only 6 scalar equations no matter the number of bodies.}

\emph{For the sake of simplicity we do not pursue this option at the present
time, as we believe that the reduction to a polynomial nonlinearity makes both the
presentation and implementation of the method more transparent. We also believe 
that the polynomial
version of the problem is more amenable to high order branch following
methods and bifurcation analysis to be pursued in a future work. We remark that,
since we work in a space of analytic functions, our argument produces useful
by-products such as bounds on coefficient decay rates, and lower bounds on the
domain of analyticity/bounds on the distances to poles in the complex plane.
This information can be used to obtain a-posteriori bounds on derivatives via
the usual Cauchy bounds of complex analysis. }
\end{remark}

The paper is organized as follows. In Section~\ref{sec:F(x)=0}, we introduce
the \emph{Fourier map} $F:X \to Y$ defined on a Banach space $X$ of
geometrically decaying Fourier coefficients, whose zeros are choreographies
having prescribed symmetries and topological properties. In
Section~\ref{sec:a-posteriori-validation}, we introduce the ideas of the
a-posteriori validation for the Fourier map, that is on how to demonstrate the
existence of true solutions of $F(x)=0$ close to numerical approximations. In
Section~\ref{sec:bounds}, we present explicit formulas for the bounds
necessary to apply the a-posteriori validation of
Section~\ref{sec:a-posteriori-validation}. We conclude the paper by presenting
the results in Section~\ref{sec:results}, where we present proofs of existence
of some spatial torus knot choreographies for $n=4,5,7$, and $9$ bodies.
The computer programs used in the paper are available at \cite{webpage}.

\section{Formulation of the problem}

\label{sec:F(x)=0}

Let $q_{j}(t)\in\mathbb{R}^{3}$ be the position in space of the body
$j\in\{1,...,n\}$ with mass $1$ at this $t$. Define the matrices
\[
\bar{I}=diag(1,1,0)\text{ and }\bar{J}=diag(J,0)\text{,}%
\]
where $J$ is the symplectic matrix in $\mathbb{R}^{2}$. In rotating
coordinates and with the period rescaled to $2\pi$,
\[
q_{j}(t)=e^{\sqrt{s_{1}}t\bar{J}}u_{j}(\omega t),
\]
the Newton equations for the $n$ bodies are%
\begin{equation}
\omega^{2}\ddot{u}_{j}+2\omega\sqrt{s_{1}}\bar{J}\dot{u}_{j}-s_{1}\bar{I}%
u_{j}=-\sum_{i=1(i\neq j)}^{n}\frac{u_{j}-u_{i}}{\left\Vert u_{j}%
-u_{i}\right\Vert ^{3}},\label{Ne}%
\end{equation}
where $\omega$ is the frequency and $s_{1}$ is defined by (\ref{s1}).

Using that $u_{j}=(w_{j},z_{j})$, the symmetries (\ref{SyP}) and (\ref{SyS})
correspond to the symmetry
\begin{equation}
u_{j}(t)=e^{\ j\bar{J}\zeta}u_{n}(t+jk\zeta)\text{.} \label{Sy}%
\end{equation}
Therefore, the solutions of the equation (\ref{Ne}) with symmetries (\ref{Sy})
are zeros of the map
\begin{equation}
\mathcal{G}(u_{n},\omega) \bydef \omega^{2}\ddot{u}_{n}+2\omega\sqrt{s_{1}}\bar
{J}\dot{u}_{n}-s_{1}\bar{I}u_{n}+\sum_{j=1}^{n-1}\frac{u_{n}-e^{\ j\bar
{J}\zeta}u_{n}(t+jk\zeta)}{\left\Vert u_{n}-e^{\ j\bar{J}\zeta}u_{n}%
(t+jk\zeta)\right\Vert ^{3}}:X\times\mathbb{R}\rightarrow Y
\label{eq:DDE_formulation}%
\end{equation}
defined in spaces $X$ and $Y$ of analytic $2\pi$-periodic functions, which we
will specify later in Fourier components. The equation $\mathcal{G}%
(u_{n},\omega)=0$, with $\mathcal{G}$ defined in \eqref{eq:DDE_formulation} is
a delay differential equation (DDE).

\subsection{Choreographies}

We say that a solution of $\mathcal{G}(u_{n},\omega )=0$, i.e. a solution of
the $n$-body problem with symmetry (\ref{Sy}), is $p:q$\emph{\ resonant }%
when it has frequency $\omega =\sqrt{s_{1}}p/q$ and (a) $kq-p=0$ or (b) $p
$ and $q$ are relatively prime and $kq-p\in n\mathbb{Z}$.  In \cite{CaDoGa18}
is proven that $p:q$ resonant orbits are choreographies in the inertial
frame, see also \cite{ChFe08}. For sake of completeness, here we reproduce a
short version of this result.

\begin{proposition}
\label{proposition} Let%
\[
Q_{j}(t)\bydef q_{j}(t/\omega )=e^{t\bar{J}\sqrt{s_{1}}/\omega }u_{j}(t)
\]%
be a reparameterization of a periodic solution in the inertial frame. An $p:q
$ resonant solution $u_{n}\ $of $\mathcal{G}(u_{n},\omega )=0$ is a
choreography in inertial frame, satisfying that $Q_{n}(t)$ is $2\pi p$%
-periodic and%
\[
Q_{j}(t)=Q_{n}(t+j\tilde{k}\zeta )\text{,}
\]%
where $\tilde{k}=k-(kq-p)\tilde{q}$ with $\tilde{q}$ the $p$-modular inverse
of $q$. The orbit of the choreography is symmetric with respect to rotations
by an angle $2\pi /p$ and the $n$ bodies form groups of $h$-polygons, where $%
h$ is the biggest common divisor of $n$ and $k$.
\end{proposition}

\begin{proof}
Since $u_{n}(t)$ is $2\pi $-periodic and $e^{t\bar{J}\sqrt{s_{1}}/\omega
}=e^{t\bar{J}q/p}$ is $2\pi p$-periodic, then the function $Q_{n}(t)=e^{t%
\bar{J}\sqrt{s_{1}}/\omega }u_{n}(t)$ is $2\pi p$-periodic. Furthermore,
since 
\begin{equation}
Q_{n}(t-2\pi )=e_{n}^{-\bar{J}2\pi q/p}Q_{n}(t),  \label{symq}
\end{equation}%
the orbit of $Q_{n}(t)$ is invariant under rotations of $2\pi /p$. The fact
that the $n$ bodies form $h$-polygons follows from symmetry (\ref{Sy}) and
the definition of $Q_{j}(t)$.

By assumption 
\[
r=(kq-p)/n\in \mathbb{Z},
\]%
then symmetry (\ref{Sy}) implies that the solution in inertial frame satisfies
\begin{equation}
Q_{j}(t)=e^{-\bar{J}2\pi j(r/p)}Q_{n}(t+jk\zeta )\text{.}  \label{qqq}
\end{equation}%
In the case (a) that $kq-p=0$, the symmetry (\ref{symq}) gives
straightforward that 
$Q_{j}(t)=Q_{n}(t+kj\zeta )$. 
In the case (b) that $p$ and $q$ are relatively prime, we can find $\tilde{q}
$ such that $q\tilde{q}=1$ mod $p$. It follows from the symmetry (\ref{symq}%
) that 
\[
Q_{n}(t-2\pi jr\tilde{q})=e^{-\bar{J}2\pi j(r/p)}Q_{n}(t).
\]%
Therefore,
\[
Q_{j}(t)=e^{-\bar{J}2\pi j(r/p)}Q_{n}(t+jk\zeta )=Q_{n}(t+j(k-rn\tilde{q}%
)\zeta ). \qedhere
\]
\end{proof}

\begin{corollary}[$(p,q)$-torus knots] \label{cor:torusKnots}
In the case that $u_{n}(t)$ is a  $p:q$ resonant orbit in the axial family
that does not cross the $z$-axis, then $Q_{n}(t)$ winds (after the period $%
2\pi p$) around a toroidal manifold with winding numbers $p$ and $q$, 
\textit{i.e.}, the choreography path is a $(p,q)$-torus knot. In the case
that $u_{n}(t)$ is a $p:q$ resonant orbit in the vertical Lyapunov family
that does not cross the $z$-axis, then the choreography $Q_{n}(t)$ winds $p$
times in a cylindrical surface. 
\end{corollary}

We conclude that the solution $q_{j}(t)=Q_{j}(t\omega)$ is a $2\pi q/\sqrt{%
s_{1}}$-periodic choreography satisfying the properties discussed above for $%
Q_{j}(t)$. Therefore, by validating solutions of $\mathcal{G}%
(u_{n},\omega)=0 $ in the axial family we prove rigorously the existence of
choreography paths that are $(p,q)$-torus knots.

\subsection{Symmetries, integrals of movement and Poincar\'{e} conditions}

Here after we omit the index $n$ that represents the $n$th body in the map
$\mathcal{G}(u)$ and denote the components of $u$ by
\[
u=(u_{1},u_{2},u_{3}).
\]
The map $\mathcal{G}(u)$ that gives the existence of choreographies is the
gradient of the action $\mathcal{A}(u):X\rightarrow\mathbb{R}$ of the $n$-body
problem reduced to paths with symmetries (\ref{Sy}). The action $\mathcal{A}%
(u)$ is invariant under the action of the group $(\theta,\varphi,\tau)\in
G\bydef T^{2}\times\mathbb{R}$ in $u\in X$ given by
\[
(\theta,\varphi,\tau)u(t)=e^{\bar{J}\theta}u(t+\varphi)+(0,0,\tau)\text{,}%
\]
which corresponds to $z$-translations and $(x,y)$-rotations of bodies, and time shift.

Given that the gradient $\mathcal{G}=\nabla\mathcal{A}$ is $G$-equivariant,
$\mathcal{G}((\theta,\varphi,\tau)u)=(\theta,\varphi,\tau)\mathcal{G}(u)$, if
$u_{0}$ is a critical point of $\mathcal{A}$, then $(\theta,\varphi,\tau)u_{0}$
is a critical point for all $(\theta,\varphi,\tau)\in G$, because
\begin{equation}
\mathcal{G}((\theta,\varphi,\tau)u_{0})=(\theta,\varphi,\tau)\mathcal{G}%
(u_{0})=0. \label{ze}%
\end{equation}
Therefore, if $u_{0}$ is not fixed by the elements of $G$, then its orbit
under the action of the group forms a $3$-dimensional manifold of zeros of
$\mathcal{G}$. Taking derivatives respect the parameters $\theta,$ $\varphi$
and $\tau$ of equation (\ref{ze}) and evaluating the parameter at $0$,\ we
obtain by the chain rule that $D\mathcal{G}(u_{0})A_{j}(u_{0})=0$, where
$A_{j}$ are the generator fields of the group $G$,
\begin{align*}
A_{1}(u)  &  =\partial_{\theta}|_{\theta=0}(\theta,0,0)u=\bar{J}u\text{,}\\
A_{2}(u)  &  =\partial_{\varphi}|_{\varphi=0}(0,\varphi,0)u=\dot{u}\text{,}\\
A_{3}(u)  &  =\partial_{\tau}|_{\tau=0}(0,0,\tau)u=(0,0,1)\text{.}%
\end{align*}
Therefore $D\mathcal{G}(u_{0})$ has the zero eigenvalues $A_{j}(u_{0})$ for
$j=1,2,3$ corresponding to tangent vectors to the $3$-dimensional manifold
generated by the action of $G$. This property holds for any equivariant field
even if it is not gradient.

In addition, for gradient maps $\mathcal{G}=\nabla\mathcal{A}$, we have also
conserved quantities generated by the action of the group $G$ (Noether
theorem). That is, since the action is invariant, $\mathcal{A}((\theta
,\varphi,\tau)u)=\mathcal{A}(u)$, deriving respect $\theta,$ $\varphi$ and
$\tau$ and evaluating the parameters at $0$, we have by chain rule that%
\begin{equation}
0=\partial_{j}\mathcal{A}(u)=\partial_{j}\mathcal{A}((\theta,\varphi
,\tau)u)=\left\langle \nabla\mathcal{A}(u),A_{j}(u)\right\rangle =\left\langle
\mathcal{G}(u),A_{j}(u)\right\rangle \text{,} \label{res}%
\end{equation}
i.e. the field $\mathcal{G}$ is orthogonal to the infinitesimal generators
$A_{j}(u)$ for $j=1,2,3$.

In summary, we have that the map $\mathcal{G}$ has $3$-dimensional families of
zeros and also $3$-restrictions given by (\ref{res}). To prove the existence
of solutions, we could take $3$-restrictions in the domain and range of
$\mathcal{G}$. But given that the range is a non-flat manifold, it is simpler
to augment the delay differential equation $\mathcal{G}=0$ with the three
Lagrangian multipliers $\lambda_{j}$ for $j=1,2,3$,
\begin{equation}
\mathcal{G}(u,\omega)+\sum_{j=1}^{3}\lambda_{j}A_{j}(u)=0.\label{DDE2}%
\end{equation}
An important observation is that the solutions of equation \eqref{DDE2} are
equivalent to the solutions of the original equations of motion.

\begin{proposition}
\label{gen}If $A_{j}(u)$ are linearly independent for $j=1,2,3$, then a
solution $u$ to $\mathcal{G}(u,\omega)=0$ is a solution to the equation
\eqref{DDE2} if and only if $\lambda_{j}=0$ for $j=1,2,3$.
\end{proposition}

\begin{proof}
Taking the product of (\ref{DDE2}) with respect to a generator $A_{j}(u)$
and using the orthogonality we obtain%
\[
\sum_{j=1}^{3}\lambda_{j}\left\langle A_{j}(u),A_{i}(u)\right\rangle =0.
\]
The result follows from the linear independence of $A_{j}(u)$, see
\cite{CaDoGa18} for details.
\end{proof}

Also the restriction in the domain forms a non-flat manifold, and it is
simpler to augment the equation (\ref{DDE2}) with three equations that
represent the respective Poincar\'e sections $I_{j}(u)=0$. Each geometric
condition $I_{j}(u)=0$ with
\[
I_{j}(u)=\left\langle u-\tilde{u},A_{j}(\tilde{u})\right\rangle :X\rightarrow
\mathbb{R}^{3}\text{,}%
\]
implies that $u$ is in the orthogonal plane to the orbit of $\tilde{u}$ under
the action of $G$, where $\tilde{u}$ is a reference solution, which typically
is the solution in the previous step of the continuation.

Taking as reference $\tilde{u}=(1,0,0)$ for the generators $A_{3}(\tilde
{u})=(0,0,1)$, then
\begin{equation}
I_{3}(u)=\int_{0}^{2\pi}u(t)\cdot(0,0,1)~dt=\int_{0}^{2\pi}u_{3}%
(t)~dt.\label{eq:phase_condition_I3}%
\end{equation}
Given a reference solution $\tilde{u}$, the other geometric conditions are
given explicitly by
\begin{equation}
I_{1}(u)=\int_{0}^{2\pi}\left(  u-\tilde{u}\right)  \cdot\bar{J}\tilde
{u}~dt=\int_{0}^{2\pi}u\cdot\bar{J}\tilde{u}~dt\label{eq:phase_condition_I1}%
\end{equation}
and
\begin{equation}
I_{2}(u)=\int_{0}^{2\pi}\left(  u-\tilde{u}\right)  \cdot\tilde{u}^{\prime
}(t)=\int_{0}^{2\pi}u(t)\cdot\tilde{u}^{\prime}(t)~dt\text{.}%
\label{eq:phase_condition_I2}%
\end{equation}

The generators $A_{j}(u)$ are linearly independent in the solutions that we
are looking. In other cases the solutions are relative equilibria, which
represents a simpler problem than the map $\mathcal{G}$.

\subsection{Automatic differentiation: obtaining a polynomial problem} \label{eq:automatic_differentiation}

Setting $\dot{u}=v$, equation $\mathcal{G}(u,\omega)=0$ becomes
\[
\omega^{2}\dot{v}+2\omega\sqrt{s_{1}}\bar{J}v-s_{1}\bar{I}u+\sum_{j=1}%
^{n-1}\frac{u-e^{\ j\bar{J}\zeta}u(t+jk\zeta)}{\left\Vert u-e^{\ j\bar{J}%
\zeta}u(t+jk\zeta)\right\Vert ^{3}}=0.
\]
In this section, we turn the non-polynomial DDE \eqref{DDE2} into a higher
dimensional DDE with polynomial nonlinearities, using the automatic
differentiation technique as in \cite{Mi18,jp_jaime_j,jp_chris_jb_DDE}. For
this, we define for $j=1,...,n-1$ the variables
\[
w_{j}(t)=\frac{1}{\left\Vert u(t)-e^{\ j\bar{J}\zeta}u(t+jk\zeta)\right\Vert
}.
\]
Then $w_{j}$ satisfy
\begin{align*}
\dot{w}_{j}  &  =\frac{d}{dt}\left(  \left\Vert u(t)-e^{\ j\bar{J}\zeta
}u(t+jk\zeta)\right\Vert ^{2}\right)  ^{-1/2}\\
&  =-w_{j}^{3}\left\langle v(t)-e^{\ j\bar{J}\zeta}v(t+jk\zeta
),u(t)-e^{\ j\bar{J}\zeta}u(t+jk\zeta)\right\rangle .
\end{align*}

Therefore, the augmented system of equations \eqref{DDE2} is
\begin{align}
\label{eq:dde1_automatic_differentiation}
\dot{u}  &  =v 
\\
\label{eq:dde2_automatic_differentiation}
\dot{v}  &  =\frac{1}{\omega^{2}}\left(  -2\omega\sqrt{s_{1}}\bar{J}%
v+s_{1}\bar{I}u-\sum_{j=1}^{n-1}w_{j}^{3}\left(  u(t)-e^{\ j\bar{J}\zeta
}u(t+jk\zeta)\right)  \right)  +\lambda_{1}\bar{J}u+\lambda_{2}v+\lambda
_{3}e_{3}
\\
\label{eq:dde3_automatic_differentiation}
\dot{w}_{j}  &  =-w_{j}^{3}\left\langle v(t)-e^{\ j\bar{J}\zeta}%
v(t+jk\zeta),u(t)-e^{\ j\bar{J}\zeta}u(t+jk\zeta)\right\rangle +\alpha
_{j}w_{j}^{3},
\end{align}
for $j=1,\dots,n-1$, where $e_{3}=(0,0,1)$. We supplement these equations with
the conditions
\begin{equation}
w_{j}(0)=\frac{1}{\left\Vert u(0)-e^{\ j\bar{J}\zeta}u(jk\zeta)\right\Vert
},\quad j=1,\dots,n-1,\label{eq:extra_initial_conditions}%
\end{equation}
which are balanced by the unfolding parameters $\alpha_{1},\dots,\alpha_{n-1}$
(e.g. see \cite{jp_jaime_j}), similarly to the manner in which the phase
conditions $I_{1}(u)=I_{2}(u)=I_{3}(u)=0$ (given respectively by
\eqref{eq:phase_condition_I1}, \eqref{eq:phase_condition_I2} and
\eqref{eq:phase_condition_I3}) are balanced by the unfolding parameters
$\lambda_{1}$, $\lambda_{2}$ and $\lambda_{3}$. Indeed, we can prove that a
solution of this system is necessarily a solution of the $n$-body problem
similarly to Proposition \ref{gen}.

\begin{proposition} \label{prop:alpha=0}
A $2\pi$-periodic solution $(u,v,w)$ of the system \eqref{eq:dde1_automatic_differentiation}-\eqref{eq:dde3_automatic_differentiation} with the conditions
(\ref{eq:extra_initial_conditions}) satisfies that $\alpha_{j}=0$ for
$j=1,...,n$, i.e. $u$ is a $2\pi$-periodic solution of $\mathcal{G}(u,\omega)=0$.
\end{proposition}

\begin{proof}
Dividing the equation for $w_{j}$ by $w_{j}^{3}$ and using that $v=\dot{u}$,
we obtain that
\[
\frac{d}{dt}\left(  -2w_{j}^{-2}\right)  =\frac{d}{dt}\left(  -\frac{1}%
{2}\left\Vert u(t)-e^{\ j\bar{J}\zeta}u(t+jk\zeta)\right\Vert ^{2}\right)
+\alpha_{j}\text{.}%
\]
Since $(u,v,w)$ is $2\pi$-periodic, integrating over the period $2\pi$, we
obtain that $2\pi\alpha_{j}=0$, see \cite{jp_jaime_j} for details. Given that
$\alpha_{j}=0$, the initial condition (\ref{eq:extra_initial_conditions})
implies that $w_{j}(t)=\left\Vert u(t)-e^{\ j\bar{J}\zeta}u(t+jk\zeta
)\right\Vert ^{-1}$. Therefore, $u$ is a solution to the augmented system
\eqref{DDE2} and, by Proposition \ref{gen}, to the equation $\mathcal{G}%
(u,\omega)=0$.
\end{proof}

In the next section, equations \eqref{eq:dde1_automatic_differentiation},
\eqref{eq:dde2_automatic_differentiation},
\eqref{eq:dde3_automatic_differentiation} and
\eqref{eq:extra_initial_conditions} are combined with Fourier expansions to
set up the \emph{Fourier map} whose zeros corresponds to choreographies having
the prescribed symmetry \eqref{Sy} and the topological property of a torus knot.

\subsection{Fourier map for automatic differentiation} \label{sec:Fourier_map}

The goal of this section is to look for periodic solutions of the delay differential equations \eqref{eq:dde1_automatic_differentiation}, \eqref{eq:dde2_automatic_differentiation} and 
\eqref{eq:dde3_automatic_differentiation} satisfying the extra conditions \eqref{eq:extra_initial_conditions} using the Fourier series expansions 
\begin{align}
\nonumber
u(t)  &  =
\begin{pmatrix}
u_{1}(t)\\
u_{2}(t)\\
u_{3}(t)
\end{pmatrix}
= \sum_{\ell\in\mathbb{Z}}e^{i \ell t}u_{\ell}, \quad u_{\ell} =
\begin{pmatrix}
(u_{1})_{\ell}\\
(u_{2})_{\ell}\\
(u_{3})_{\ell}%
\end{pmatrix}
\\
v(t)  &  =%
\begin{pmatrix}
v_{1}(t)\\
v_{2}(t)\\
v_{3}(t)
\end{pmatrix}
=\sum_{\ell\in\mathbb{Z}}e^{i \ell t}v_{\ell}, \quad v_{\ell} =
\begin{pmatrix}
(v_{1})_{\ell}\\
(v_{2})_{\ell}\\
(v_{3})_{\ell}%
\end{pmatrix}
\label{eq:Fourier_expansions}
\\
w(t)  &  =%
\begin{pmatrix}
w_{1}(t)\\
\vdots\\
w_{n-1}(t)
\end{pmatrix}
=\sum_{\ell\in\mathbb{Z}}e^{i \ell t}w_{\ell}, \quad w_{\ell} =
\begin{pmatrix}
(w_{1})_{\ell}\\
\vdots\\
(w_{n-1})_{\ell}
\end{pmatrix}.
\nonumber
\end{align}

Based on the fact that periodic solutions of analytic DDEs are analytic \cite{MR0322315}, we consider the following Banach space of geometrically decaying Fourier coefficients
\begin{equation} \label{eq:ell_nu_one}
\ell_\nu^1 \bydef \left\{ c = (c_\ell)_{\ell \in \Z} : \| c\|_\nu \bydef \sum_{\ell \in \Z} |c_\ell | \nu^{|\ell|} < \infty \right\},
\end{equation}
where $\nu \ge 1$. If $\nu>1$ and $a=(a_\ell)_{\ell \in \Z} \in \ell_\nu^1$, then the function $t \mapsto \sum_{\ell \in\mathbb{Z}} e^{i \ell t}a_{\ell}$ defines a $2\pi$-periodic analytic function on the complex strip of width $\ln(\nu)>0$. Another useful property of the space $\ell_\nu^1$ is that it is a Banach algebra under discrete convolution $*:\ell_\nu^1 \times \ell_\nu^1 \to \ell_\nu^1$ defined as
\[
(a*b)_k = \sum_{k_1+k_2=k} a_{k_1} b_{k_2},
\]
where $a,b \in \ell_\nu^1$. More explicitly, $\| a * b \|_\nu \le \| a \|_\nu \| b \|_\nu$, for all $a,b \in \ell_\nu^1$ and $\nu \ge 1$.

The unknowns of the DDEs \eqref{eq:dde1_automatic_differentiation}, \eqref{eq:dde2_automatic_differentiation} and 
\eqref{eq:dde3_automatic_differentiation} are given by the unfolding parameters $\lambda \bydef (\lambda_j)_{j=1}^3 \in \C^3$ and $\alpha \bydef (\alpha_j)_{j=1}^{n-1} \in \C^{n-1}$, and the Fourier coefficients $u = (u_j)_{j=1}^3 \in (\ell_\nu^1)^3$, $v = (v_j)_{j=1}^3 \in (\ell_\nu^1)^3$ and $w = (w_j)_{j=1}^{n-1} \in (\ell_\nu^1)^{n-1}$.
The total vector of unknown $x$ and the Banach space $X$ are then given by
\begin{equation} \label{eq:Banach_space_X}
x \bydef
\begin{pmatrix}
\lambda\\
\alpha\\
u\\
v\\
w
\end{pmatrix}
\in X \bydef
\mathbb{C}^{3} \times\mathbb{C}^{n-1} \times(\ell_{\nu}^{1})^{3} \times
(\ell_{\nu}^{1})^{3} \times(\ell_{\nu}^{1})^{n-1} \cong \mathbb{C}^{n+2}
\times(\ell_{\nu}^{1})^{n+5}.
\end{equation}
The Banach space $X$ is endowed with the norm
\begin{equation} \label{eq:Banach_space_X}
\|x\|_{X} \bydef 
\max\left\{
|\lambda|_{\infty}, |\alpha|_{\infty},\max_{j=1,2,3}\|u_{j}\|_{\nu},
\max_{j=1,2,3}\|v_{j}\|_{\nu},\max_{j=1,\dots,n-1}\|w_{j}\|_{\nu} \right\} ,
\end{equation}
where
\[
|\lambda|_{\infty}= \max_{j=1,2,3} |\lambda_{j}| \quad\text{and} \quad
|\alpha|_{\infty}= \max_{j=1,\dots,n-1} |\alpha_{j}|.
\]

In order to define the Fourier map problem $F(x)=0$, we plug the Fourier expansions \eqref{eq:Fourier_expansions} in \eqref{eq:dde1_automatic_differentiation}, \eqref{eq:dde2_automatic_differentiation}, \eqref{eq:dde3_automatic_differentiation} and \eqref{eq:extra_initial_conditions}, and solve for the corresponding nonlinear map. First note that
\begin{align*}
u(t)-e^{\ j\bar{J}\zeta}u(t+jk\zeta)  &  =\sum_{\ell\in\mathbb{Z}}\left(
u_{\ell} -e^{\ j\bar{J}\zeta}e^{ijk \ell\zeta}u_{\ell}\right)  e^{i \ell
t}=\sum_{\ell\in\mathbb{Z} }M_{j \ell}u_{\ell}e^{i\ell t}\text{,}%
\end{align*}
where $M_{j\ell}$ is defined as
\[
M_{j\ell}=I-e^{\ j\bar{J}\zeta}e^{ijk\ell\zeta} =
\begin{pmatrix}
1- e^{ijk\ell\zeta} \cos(j \zeta) & e^{ijk\ell\zeta} \sin(j \zeta) & 0\\
- e^{ijk\ell\zeta} \sin(j \zeta) & 1- e^{ijk\ell\zeta} \cos(j \zeta) & 0\\
0 & 0 & 1 - e^{ijk\ell\zeta}%
\end{pmatrix},
\]
since $\bar{J}=diag(J,0)$ with $J=%
\begin{pmatrix}
0 & -1\\
1 & 0
\end{pmatrix}
$.

In Fourier space, the phase conditions $I_{1}(u)=I_{2}(u)=I_{3}(u)=0$ (see
\eqref{eq:phase_condition_I1}, \eqref{eq:phase_condition_I2} and
\eqref{eq:phase_condition_I3}, respectively) are given by%

\begin{align*}
I_{1}(u)  &  = \int_{0}^{2\pi} -u_{1}(t) \tilde u_{2}(t) + u_{2}(t) \tilde u_{1}(t)
~{dt}\\
&  =-(u_{1} * \tilde u_{2})_{0} + (u_{2} * \tilde u_{1})_{0}\\
&  = \sum_{\ell\in\mathbb{Z}} - (u_{1})_{\ell} (\tilde u_{2})_{-\ell} +
(u_{2})_{\ell} (\tilde u_{1})_{-\ell}\\
I_{2}(u) & =\int_{0}^{2\pi} \left(  u_{1}(t) \tilde u_{1}^{\prime}(t) + u_{2}(t)
\tilde u_{2}^{\prime}(t) + u_{3}(t) \tilde u_{3}^{\prime}(t) \right) ~dt\\
&  = (u_{1}*\tilde u_{1}^{\prime})_{0} + (u_{2}*\tilde u_{2}^{\prime})_{0} +
(u_{3}*\tilde u_{3}^{\prime})_{0}\\
&  = \sum_{\ell\in\mathbb{Z}} i \ell\left(  (u_{1})_{\ell} (\tilde
u_{1})_{-\ell} + (u_{2})_{\ell} (\tilde u_{2})_{-\ell} + (u_{3})_{\ell}
(\tilde u_{3})_{-\ell} \right) \\
I_{3}(u) &  = \int_{0}^{2\pi}u_{3}(t)~dt = (u_{3})_{0},
\end{align*}
where $\tilde u_{1}$, $\tilde u_{2}$ and $\tilde u_{3}$ have only finitely many non zero terms.

Hence, setting $\eta: (\ell_{\nu}^{1})^{3} \to\mathbb{C}^{3}$  as
\begin{equation}
\label{eq:eta}
\eta(u) =
\begin{pmatrix}
\eta_{1}(u)\\
\eta_{2}(u)\\
\eta_{3}(u)
\end{pmatrix}
\bydef
\begin{pmatrix}
-(u_{1} * \tilde u_{2})_{0} + (u_{2} * \tilde u_{1})_{0}\\
(u_{1}*\tilde u_{1}^{\prime})_{0} + (u_{2}*\tilde u_{2}^{\prime})_{0} +
(u_{3}*\tilde u_{3}^{\prime})_{0}\\
(u_{3})_{0}%
\end{pmatrix}
,
\end{equation}
we get that $\eta(u) = 0$ implies that $I_{1}(u)=I_{2}(u)=I_{3}(u)=0$.
Given $j=1,\dots,n-1$ and $u \in(\ell_{\nu}^{1})^{3}$, denote $M_{j} u
\in(\ell_{\nu}^{1})^{3}$ component-wise by
\[
(M_{j} u )_{\ell}
\bydef
M_{j \ell}
u_{\ell}=
\begin{pmatrix}
\left(  M_{j \ell} u_{\ell}\right) _{1}\\
\left(  M_{j \ell} u_{\ell}\right) _{2}\\
\left(  M_{j \ell} u_{\ell}\right) _{3}
\end{pmatrix} 
=
\begin{pmatrix}
\left(  1- e^{ijk\ell\zeta} \cos(j \zeta) \right)  (u_{1})_{\ell}+
e^{ijk\ell\zeta} \sin(j \zeta) (u_{2})_{\ell}\\
- e^{ijk\ell\zeta} \sin(j \zeta) (u_{1})_{\ell}+ \left(  1- e^{ijk\ell\zeta}
\cos(j \zeta) \right)  (u_{2})_{\ell}\\
\left( 1 - e^{ijk\ell\zeta} \right)  (u_{3})_{\ell}%
\end{pmatrix}.
\]
In Fourier space, the extra initial condition
\eqref{eq:extra_initial_conditions} (given $j=1,\dots,n-1$) is simplified as
\begin{align}
\label{eq:extra_initial_conditions_Fourier}\gamma_{j}(u,w_{j})  &
\bydef w_{j}(0)^{2} \left\|  \sum_{\ell\in\mathbb{Z}} M_{j \ell}u_{\ell} \right\|
^{2}-1\nonumber  = \left(  \sum_{\ell\in\mathbb{Z}} (w_{j})_{\ell}\right) ^{2} \left[
\sum_{p=1}^{3} \left(  \sum_{\ell\in\mathbb{Z}} (M_{j \ell}u_{\ell})_{p}
\right) ^{2} \right]  - 1.
\end{align} 
Set $\gamma: (\ell_{\nu}^{1})^{3} \times(\ell_{\nu}^{1})^{n-1} \to\mathbb{C}^{n-1}$ as 
\begin{equation}
\label{eq:gamma_initial_conditions}
\gamma(u,w) \bydef
\begin{pmatrix}
\gamma_{1}(u,w_{1})\\
\gamma_{2}(u,w_{2})\\
\vdots\\
\gamma_{n-1}(u,w_{n-1})
\end{pmatrix}.
\end{equation}
Hence, $\gamma(u,w) = 0$ implies that \eqref{eq:extra_initial_conditions} holds.

For sake of simplicity of the presentation, given any $N\in\mathbb{N}$, denote
the differentiation operator $D$ acting on $u \in (\ell_\nu^1)^N$ as
\begin{equation} \label{eq:differentiation_Fourier}
(Du)_{\ell} \bydef i \ell u_{\ell} 
=
\begin{pmatrix}
i \ell(u_{1})_{\ell}\\
i \ell(u_{2})_{\ell}\\
\vdots\\
i \ell(u_{N})_{\ell}%
\end{pmatrix}.
\end{equation}

\begin{remark}
The linear operator $D$ is not bounded on $(\ell_\nu^1)^N$. However, it is bounded when considering the image to be slightly less regular. More explicitly, letting 
\begin{equation} \label{eq:tilde_ell_nu_one}
\tilde \ell_\nu^1 \bydef \left\{ c = (c_\ell)_{\ell \in \Z} : |c_0 | + \sum_{\ell \ne 0} |c_\ell | \frac{\nu^{|\ell|}}{|\ell|} < \infty \right\},
\end{equation}
we can easily verify that $D:(\ell_\nu^1)^N \to (\tilde \ell_\nu^1)^N$ is a bounded linear operator.
\end{remark}

Let $f: (\ell_{\nu}^{1})^{3} \times(\ell_{\nu}^{1})^{3} \to(\tilde \ell_{\nu
}^{1})^{3}$ be defined by
\begin{equation}
\label{eq:f_Fourier_map_automatic_differentiation}
f(u,v) \bydef D u - v.
\end{equation}
Note that $f(u,v)=0$ ensures that \eqref{eq:dde1_automatic_differentiation} holds.
Let $g: \C^3 \times (\ell_{\nu}^{1})^{3} \times(\ell_{\nu}^{1})^{3} \times(\ell_{\nu}%
^{1})^{n-1} \times \C \to(\tilde \ell_{\nu}^{1})^{3}$ be defined by
%
%\begin{equation} \label{eq:g_Fourier_map_automatic_differentiation}
%g(u,v,w,\omega)\bydef  D v - \frac{1}{\omega^{2}} \left( -2\omega\sqrt{s_{1}}\bar{J}v+s_{1}\bar
%{I} u - \sum_{j=1}^{n-1} M_j (u * w_j^3) \right) - \lambda_{1}\bar{J}u - \lambda_{2}v - \lambda_{3} \hat e_{3},
%\end{equation}
%
\begin{equation}
\label{eq:g_Fourier_map_automatic_differentiation}
g(\lambda,u,v,w,\omega)
\bydef
\omega^{2} D v + 2\omega\sqrt{s_{1}}\bar{J}v - s_{1}\bar{I} u + \lambda
_{1}\bar{J}u + \lambda_{2}v + \lambda_{3} \hat e_{3} + \sum_{j=1}^{n-1} 
(M_{j} u) * w_{j}^{3},
\end{equation}
where $ (M_{j} u) * w_{j}^{3} \in(\ell_{\nu}^{1})^{3}$ is given component-wise by
\[
\left(  (M_{j} u) * w_{j}^{3} \right) _{\ell}\,\overset
{\mbox{\tiny\textnormal{\raisebox{0ex}[0ex][0ex]{def}}}}{=}\, 
\begin{pmatrix}
\left( (M_{j} u)_{1} * w_{j}^{3} \right)_{\ell}\\
\left( (M_{j} u)_{2} * w_{j}^{3} \right)_{\ell}\\
\left( (M_{j} u)_{3} * w_{j}^{3} \right)_{\ell}%
\end{pmatrix},
\]
and where $\hat e_{3} \in(\ell_{\nu}^{1})^{3}$ is given component-wise by
\[
(\hat e_{3})_{\ell}\,\overset
{\mbox{\tiny\textnormal{\raisebox{0ex}[0ex][0ex]{def}}}}{=}\,
\begin{pmatrix}
0\\
0\\
\delta_{\ell,0}%
\end{pmatrix}
,
\]
with $\delta_{i,j}$ being the Kronecker delta. Note that $g(\lambda,u,v,w,\omega)=0$ ensures
that \eqref{eq:dde2_automatic_differentiation} holds.

Let $h_{j}: \C \times (\ell_{\nu}^{1})^{3} \times(\ell_{\nu}^{1})^{3} \times\ell_{\nu
}^{1} \to \tilde \ell_{\nu}^{1}$ be defined by
\begin{equation}
\label{eq:h_j_Fourier_map_automatic_differentiation}
h_{j}(\alpha_j,u,v,w_{j})
\bydef
D w_{j} + w_{j}^{3}* \left(  \sum_{p=1}^{3} (M_{j} u)_{p} * (M_{j} v)_{p}
\right)  + \alpha_{j} w_{j}^{3}%
\end{equation}
and let $h: \C^{n-1} \times (\ell_{\nu}^{1})^{3} \times(\ell_{\nu}^{1})^{3} \times(\ell_{\nu
}^{1})^{n-1} \to(\tilde \ell_{\nu}^{1})^{n-1}$ be defined by
\begin{equation}
\label{eq:h_Fourier_map_automatic_differentiation}
h(\alpha,u,v,w)
\bydef
\begin{pmatrix}
h_{1}(\alpha_1,u,v,w_{1})\\
h_{2}(\alpha_2,u,v,w_{2})\\
\vdots\\
h_{n-1}(\alpha_{n-1},u,v,w_{n-1})
\end{pmatrix}.
\end{equation}
Hence, $h(\alpha,u,v,w)=0$ implies that \eqref{eq:dde3_automatic_differentiation} hold.

Defining 
\begin{equation} \label{eq:Banach_space_Y}
Y \bydef
\mathbb{C}^{3} \times\mathbb{C}^{n-1} \times(\tilde \ell_{\nu}^{1})^{3} \times
(\tilde \ell_{\nu}^{1})^{3} \times(\tilde \ell_{\nu}^{1})^{n-1}
\end{equation}
the Fourier map $F:X \times \R \to Y$ is defined by 
%let
%\[
%x \bydef
%\begin{pmatrix}
%\lambda\\
%\alpha\\
%u\\
%v\\
%w
%\end{pmatrix}
%\in X \bydef
%\mathbb{C}^{3} \times\mathbb{C}^{n-1} \times(\ell_{\nu}^{1})^{3} \times
%(\ell_{\nu}^{1})^{3} \times(\ell_{\nu}^{1})^{n-1}
%\]
%and recalling \eqref{eq:eta}, \eqref{eq:gamma_initial_conditions},
%\eqref{eq:f_Fourier_map_automatic_differentiation},
%\eqref{eq:g_Fourier_map_automatic_differentiation} and
%\eqref{eq:h_Fourier_map_automatic_differentiation}, we get to define the final
%
\begin{equation}
\label{eq:F_Fourier_map_automatic_differentiation}
F(x,\omega) 
\bydef
\begin{pmatrix}
\eta(u)\\
\gamma(u,w)\\
f(u,v)\\
g(\lambda,u,v,w,\omega)\\
h(\alpha,u,v,w)
\end{pmatrix}.
\end{equation}
For a fixed $\omega>0$, we introduce in Section~\ref{sec:a-posteriori-validation} an a-posteriori 
validation method for the Fourier map, that is we develop a systematic and constructive approach to prove existence of $x \in X$ 
such that $F(x,\omega) =0$. By construction, the solution $x$ yields a choreography having the prescribed symmetry \eqref{Sy} 
and the topological property of a torus knot. 

\section{A-posteriori validation for the Fourier map} \label{sec:a-posteriori-validation}

The idea of the computer-assisted proof of existence of a spatial torus-knot choreography is to demonstrate that a certain Newton-like operator is a contraction on a closed ball centered at a numerical approximation $\bx$. To compute $\bx$, we consider a finite dimensional projection of the Fourier map $F: X \times \R \to Y$. Given a number $m \in \N$, and given a vector $a = (a_\ell)_{\ell \in \Z} \in \ell_\nu^1$, consider the projection
\begin{align*}
\pi^m : \ell_\nu^1& \to \C^{2m-1} \\ 
a &\mapsto \pi^m a \bydef (a_\ell)_{|\ell|<m} \in \C^{2m-1}.
\end{align*}
We generalize that projection to get $\pi_N^m:(\ell_\nu^1)^{N} \to \C^{N(2m-1)}$ defined by
\[
\pi_N^m(a^{(1)}, \dots, a^{(N)}) \bydef ( \pi^m a^{(1)},\dots,\pi^m a^{(N)}) \in \C^{N(2m-1)}
\]
and $\Pi^{(m)}: X  \to \C^{2m(n + 5) - 3}$ defined by
\[
\Pi^{(m)} x = \Pi^{(m)} (\lambda,\alpha,u,v,w) \bydef \left( \lambda,\alpha,\pi_3^m u, \pi_3^m v,\pi_{n-1}^mw \right) \in \C^{2m(n + 5) - 3}.
\]
Often, given $x \in X$, we denote
\[
x^{(m)} \bydef \Pi^{(m)} x \in \C^{2m(n + 5) - 3}.
\]
Moreover, we define the natural inclusion $\iota^m : \C^{2m-1} \xhookrightarrow{} \ell_\nu^1$ as follows. For $a = (a_\ell)_{|\ell|<m} \in \C^{2m-1}$ let $\iota^m a \in \ell_\nu^1$ be defined component-wise by
\[
\left( \iota^m a \right)_{\ell}
= 
\begin{cases}
a_\ell, & |\ell|<m
\\
0, & |\ell| \ge m.
\end{cases}
\]
Similarly, let $\iota_N^m:\C^{N(2m-1)} \xhookrightarrow{} (\ell_\nu^1)^{N}$ be the natural inclusion defined as follows. Given $a = (a^{(1)},\dots,a^{(N)}) \in (\C^{2m-1})^N \cong \C^{N(2m-1)}$, 
\[
\iota_N^m a \bydef \left( \iota^m a^{(1)}, \dots,\iota^m a^{(N)} \right) \in (\ell_\nu^1)^{N}.
\]
Finally, let the natural inclusion $\inc : \mathbb{C}^{2m(n + 5) - 3} \xhookrightarrow{} X$ be defined, for $x \in \C^{2m(n + 5) - 3}$ as 
\[
\inc x = \inc(\lambda,\alpha,u,v,w) \bydef \left( \lambda,\alpha,\iota_3^m u, \iota_3^m v,\iota_{n-1}^mw \right) \in X.
\]
Finally, let the {\em finite dimensional projection} $F^{(m)} : \mathbb{C}^{2m(n + 5) - 3} \to\mathbb{C}^{2m(n + 5) - 3}$ of the Fourier map be defined, for $x \in \C^{2m(n + 5) - 3}$, as 
\begin{equation} \label{eq:F_Fourier_map_projection}
F^{(m)}(x,\omega) = \Pi^{(m)} F(\inc x,\omega).
\end{equation}
Also denote $F^{(m)} = \left( \eta^{(m)},\gamma^{(m)},f^{(m)},g^{(m)},h^{(m)} \right)$.

Assume that, using Newton's method, a numerical approximation $\bx \in \mathbb{C}^{2m(n + 5) - 3}$ of \eqref{eq:F_Fourier_map_projection} has been obtained at a parameter (frequency) value $\omega$, that is $F^{(m)}(\bx,\omega) \approx 0$. We slightly abuse the notation and denote $\bx \in \mathbb{C}^{2m(n + 5) - 3}$ and $\inc \bx \in X$ both using $\bx$.

We now fix an $\omega_0 \in \mathbb{R}$ and consider the mapping
$F \colon X \to Y$ defined by $F(x) = F(x, \omega_0)$.
The following result is a Newton-Kantorovich theorem with a smoothing 
approximate inverse. It provides an a-posteriori validation method for 
proving rigorously the existence of a point $\tx$ such that $F(\tx)=0$ 
and $\| \tx - \bx\|_X \le r$ for a small radius $r$. Recalling the norm on 
$X$ given in \eqref{eq:Banach_space_X}, denote by 
\[
B_{r} (y) \bydef \left\{ x \in X : \| x - y \|_X \le r \right\} \subset X
\]
the ball of radius $r$ centered at $y \in X$.

\begin{theorem}[\bf Radii Polynomial Approach] \label{thm:radPolyBanach}
For $\bar x \in X$ and $r > 0$  assume that 
$F: X \to Y$ is Fr\'echet differentiable on the ball $B_r(\bar x)$.
Consider bounded linear operators $A^{\dagger} \in B(X,Y)$ and $A \in B(Y,X)$, 
where $A^{\dagger}$ is an approximation of $D F(\bx)$ and $A$
is an approximate inverse of $D F(\bx)$. 
Observe that 
\begin{equation} \label{eq:AF:X->X}
A F \colon X  \to X.
\end{equation}
Assume that $A$ is injective.
Let $Y_0, Z_0,Z_1,Z_2 \ge 0$ be bounds satisfying
\begin{align}
\label{eq:Y0_radPolyBanach}
\| A F(\bx) \|_X &\le Y_0,
\\
\label{eq:Z0_radPolyBanach}
\| I - A A^{\dagger}\|_{B(X)} &\le Z_0,
\\
\label{eq:Z1_radPolyBanach}
\| A[D F(\bx) - A^{\dagger} ] \|_{B(X)} &\le Z_1,
\\
\label{eq:Z2_radPolyBanach}
\| A[D F(\bx+b) - D F(\bx)]\|_{B(X)} &\le Z_2 r, \quad \forall~ b \in B_r(0).
\end{align}
Define the radii polynomial
\begin{equation} \label{eq:radii_polynomial}
p(r) \bydef Z_2 r^2 + (Z_1 + Z_0 - 1) r + Y_0.
\end{equation}
If there exists $0 < r_0 \leq r$ such that
\begin{equation} \label{eq:p(r0)<0}
p(r_0) < 0,
\end{equation}
then there exists a unique $\tx \in B_{r_0} (\bx)$ such that $F(\tx) = 0$.
\end{theorem}

\begin{proof}
Details of the elementary proof are found in Appendix A of \cite{MR3612178}. 
The idea is to first show that $T(x) \bydef x - A F(x)$ satisfies $T(B_{r_0} (\bx)) \subset B_{r_0} (\bx)$,
and then to show the existence of $\kappa<1$ such that $\| T (x)-T(y)\|_X \le \kappa \|x-y\|_X$ 
for all $x,y \in B_{r_0} (\bx)$. 
These facts follow from the inequalities of Equations \eqref{eq:Y0_radPolyBanach}, 
\eqref{eq:Z0_radPolyBanach}, \eqref{eq:Z1_radPolyBanach},  \eqref{eq:Z2_radPolyBanach},
and from the hypothesis that $p(r_0) < 0$.  
The proof then follows from the contraction mapping theorem and 
the injectivity of $A$. 
\end{proof}

\noindent The following corollary provides an additional useful byproduct.

\begin{corollary}[{\bf Non-degeneracy at the true solution}] \label{cor:nonDegen}
Given the hypotheses of Theorem \ref{thm:radPolyBanach}, the linear operator 
$ADF(\tilde x)$ is boundedly invertible with
\[
  \| \left[A DF(\tilde x)\right]^{-1} \|_{B(X)} \leq \frac{1}{1 - ( Z_2 r_0 + Z_1 + Z_0)}.
\] 
\end{corollary}
\begin{proof}
From
\[
p(r_0) < 0,
\]
we obtain 
\[
Z_2 r_0^2 + (Z_1 + Z_0)r_0 + Y_0 < r_0,
\]
or
\[
Z_2 r_0 + (Z_1 + Z_0) + \frac{Y_0}{r_0} < 1.
\]
Since $Y_0$ and $r_0$ are both positive it follows that 
\[
Z_2 r_0 + (Z_1 + Z_0) < 1.
\]
Since $\tilde x \in B_{r_0}(\bar x)$ we have that $\tilde x = \bar x + b$
for some $b \in B_{r_0}(0)$, and by applying the 
inequalities of Equations 
\eqref{eq:Z0_radPolyBanach}, \eqref{eq:Z1_radPolyBanach}, and 
\eqref{eq:Z2_radPolyBanach} we have that 
\begin{align*}
\| \mbox{Id} - ADF(\tilde x) \|_{B(X)}  & \leq 
\| A (DF(\bar x + b) - DF(\bar x)) \|  
+ \| A (A^\dagger - DF(\bar x)) \|
+  \| \mbox{Id} - AA^\dagger \|\\
& \leq   Z_2 r_0 + Z_1 + Z_0  \\
& < 1.
\end{align*}
Then 
\[
A DF(\tilde x) = \mbox{Id} - \left( \mbox{Id} - ADF(\tilde x) \right), 
\]
is invertible by the Neuman theorem and
\[
\| \left[A DF(\tilde x)\right]^{-1} \| \leq \frac{1}{1 - ( Z_2 r_0 + Z_1 + Z_0)},
\]
as desired.
\end{proof}

Returning to the parameter dependent problem, suppose that $\tilde x$ is 
a zero of $F(x) = F(x, \omega_0)$ and that $A DF(\tilde x) = A D_x F(\tilde x, \omega_0)$ 
is boundedly invertible as above.
Notice that $F(x, \omega)$ is differentiable with respect to $\omega$ near $\omega_0$.
Define the mapping $G(x, \omega) = AF(x, \omega)$ and observe that 
$G$ and $F$ have the same zero set as $A$ is injective.  Observe also that 
$D_x G(x, \omega) = AD_x F(x, \omega)$.  
So $(\tilde x, \omega_0)$ is a zero of $G$ with $D_x G(\tilde x, \omega_0)$ an isomorphism, 
it follows from the implicit function theorem that $G$ has a smooth branch of 
zeros through $\tilde x$.  More precisely there exists an $\epsilon > 0$ and a smooth 
function $x \colon (\omega_0-\epsilon, \omega_0+\epsilon) \to X$ with $x(\omega_0) = \tilde x$ 
and
\[
G(x(\omega), \omega) = 0,
\]
for all $\omega \in (\omega_0-\epsilon, \omega_0+\epsilon)$.  It follows again from the injectivity of $A$
that $F(x(\omega), \omega) = 0$ for all $\omega \in (\omega_0-\epsilon, \omega_0+\epsilon)$.  
Finally, as discussed in the introduction, we obtain that
for any rational number $\sqrt{s_{1}}p/q \in (\omega_0-\epsilon, \omega_0+\epsilon)$ , the 
solution
$x(\sqrt{s_{1}}p/q)$ produces spatial torus knot choreography orbit near $\tilde x$ by proposition 4.
Taken together the results of this section show that our method produces the 
existence of countably many spatial torus knot choreographies as soon as 
Theorem \ref{thm:radPolyBanach}
succeeds at a given $\omega_0$.

\subsection{Isolated solutions yield real periodic solutions}

In this short section, we show how the output $\tx \in B_{r_0}(\bx)$ of Theorem~\ref{thm:radPolyBanach} (if any) yields a real periodic solution, provided the numerical approximation is chosen to represent a real periodic solution.

Define the operator $\sigma:\ell_\nu^1 \to \ell_\nu^1$ by $(\sigma(a))_\ell \bydef a_{-\ell}^*, \quad \text{for } \ell \in \Z$, where $z^*$ denotes the complex conjugate of $z \in \C$.
Define the symmetry subspace $\ell_\nu^{1,\real} \subset \ell_\nu^1 $ by
\[
\ell_\nu^{1,\real} \bydef \left\{ c \in \ell_\nu^1 :\sigma(c) = c \right\}.
\]
Note that if $(u_\ell)_{\ell \in \Z} \in \ell_\nu^{1,\real}$, then the function 
$u(t) \bydef \sum_{\ell \in \Z} u_\ell e^{i \ell t}$ is a real $2\pi$-periodic function. Define the operator $\Sigma: X \to X$ acting on $x = (\lambda, \alpha, u, v, w) \in X$ as 
\[
\Sigma(x) = 
\left(
\lambda^*,\alpha^*, \sigma(u_1),\sigma(u_2),\sigma(u_3), \sigma(v_1),\sigma(v_2),\sigma(v_3), \sigma(w_1),\dots,\sigma(w_{n-1}) 
\right),
\] 
where $\lambda^*  \in \C^3$ and $\alpha^* \in \C^{n-1}$ denote the component-wise complex conjugate of $\lambda \in \C^3$ and $\alpha \in \C^{n-1}$, respectively. Define the subspace $X_\real \subset X$ as
\begin{equation} \label{eq:sym_Banach_space}
X_\real \bydef \left\{ x \in X : \Sigma(x) =x \right\}.
\end{equation}
It follows by definition that $X_\real= \R^{n+2} \times(\ell_{\nu}^{1,\real})^{n+5}$.
\begin{proposition} \label{prop:real_solutions}
Fix a frequency $\omega>0$ and assume that the numerical approximation denoted $\bx = (\bar \lambda,\bar \alpha, \bu, \bv,\bw)$ satisfies $\bx \in X_{\real}$ and that the reference solution $\tu = (\tu_1,\tu_2,\tu_3)$ satisfies $\tu \in (\ell_\nu^{1,\real})^3$. Assume that there exists a unique $x \in B_r(\bx)$ such that $F( x,\omega)=0$. Then $x \in X_{\real}$.
\end{proposition}
\begin{proof}
Denote the solution $x = (\lambda,\alpha,u,v,w) \in B_{r}(\bx)$. The proof is twofold: (1) show that $F(\Sigma(x),\omega)=0$; and (2) show that $\Sigma(x) \in B_{r}(\bx)$. The conclusion $\Sigma(x) = x$ (that is $x \in X_{\real}$) then follows by unicity of the solution.  First, we have that $F(\Sigma(x),\omega)= \Sigma \left( F(x,\omega) \right)$, since the operator  $F$ corresponds to the complex extension of a real equation. Since $F(x,\omega)=0$, then $F(\Sigma(x),\omega) = \Sigma \left( F(x,\omega) \right) = \Sigma \left( 0 \right) = 0$. Second, to prove that $\Sigma(x) \in B_{r}(\bx)$, it is sufficient to realize that $|z^*| = |z|$ and that given any $c \in \ell_\nu^1$,
\begin{equation} \label{eq:sigma_invariance_under_norm}
\| \sigma(c) \|_{\nu} = \sum_{\ell \in \Z} |\sigma(c)_\ell| \nu^{|\ell|} 
 = \sum_{\ell \in \Z} |c^*_{-\ell}| \nu^{|\ell|} =  \sum_{\ell \in \Z} |c_{\ell}| \nu^{|\ell|} = \| c \|_{\nu},
\end{equation}
which shows that for any $\xi \in X$, $\| \Sigma(\xi) \|_X = \| \xi \|_X$. Hence, since $\Sigma(\bx)=\bx$, we conclude that  
\[
\| \Sigma(x) - \bx \|_X =  \| \Sigma(x) - \Sigma(\bx) \|_X = \| \Sigma(x - \bx) \|_X = \| x - \bx \|_X \le r. \qedhere
\]
\end{proof}

\subsection{Definition of the operators \boldmath$A^\dagger$\unboldmath~and \boldmath$A$\unboldmath} \label{sec:operators_A_dagA}

To apply the radii polynomial approach of Theorem~\ref{thm:radPolyBanach}, we need to define the approximate derivative 
$A^\dagger$ and the smoothing approximate inverse $A$. Consider the finite dimensional projection $F^{(m)} : \mathbb{C}^{2m(n + 5) - 3} \to\mathbb{C}^{2m(n + 5) - 3}$ and assume that at a fixed frequency $\omega>0$ we computed $\bx \in \mathbb{C}^{2m(n + 5) - 3}$ such that $F^{(m)}(\bx,\omega) \approx 0$. Denote by $DF^{(m)}(\bx,\omega) \in M_{2m(n + 5) - 3}(\C)$ the Jacobian matrix of $F^{(m)}$ at $(\bx,\omega)$. Given $x \in X$, define 
\begin{equation} \label{eq:A_dagger}
\dagA x = \inc \Pi^{(m)} \dagA x + (I- \inc \Pi^{(m)}) \dagA x,
\end{equation}
where $\Pi^{(m)} \dagA x = DF^{(m)}(\bx,\omega) x^{(m)}$ and 
\[
(I- \inc \Pi^{(m)}) \dagA x = 
\begin{pmatrix} 0 \\ 0 \\ 
(I - \iota_3^m \pi_3^m) Du
\\
\omega^2 (I - \iota_3^m \pi_3^m) Dv
\\
(I - \iota_{n-1}^m \pi_{n-1}^m) Dw
\end{pmatrix}.
\]
Recalling the definition of the Banach space $Y$ in \eqref{eq:Banach_space_Y}, we can verify that the operator $\dagA:X \to Y$ is a bounded linear operator, that is $\dagA \in B(X,Y)$. For $m$ large enough, it acts as an approximation of the true Fr\'echet derivative $D_xF(\bx,\omega)$. Its action on the finite dimensional projection is the Jacobian matrix (the derivative) of $F^{(m)}$ at $(\bx,\omega)$ while its action on the tail keeps only keep the unbounded terms involving the differentiation $D$ as defined in \eqref{eq:differentiation_Fourier}.

Consider now a matrix $A^{(m)} \in M_{2m(n + 5) - 3}(\C)$ computed so that $A^{(m)} \approx {DF^{(m)}(\bx,\omega)}^{-1}$. In other words, this means that $\| I - A^{(m)} DF^{(m)}(\bx,\omega) \| \ll 1$. This step is performed using a numerical software ({\tt MATLAB} in our case). We decompose the matrix $A^{(m)}$ block-wise as
\[
A^{(m)}=
\begin{pmatrix}
A^{(m)}_{\lambda,\lambda} & A^{(m)}_{\lambda,\alpha} & A^{(m)}_{\lambda,u} & A^{(m)}_{\lambda,v} & A^{(m)}_{\lambda,w}
\\
A^{(m)}_{\alpha,\lambda} & A^{(m)}_{\alpha,\alpha} & A^{(m)}_{\alpha,u} & A^{(m)}_{\alpha,v} & A^{(m)}_{\alpha,w}
\\
A^{(m)}_{u,\lambda} & A^{(m)}_{u,\alpha} & A^{(m)}_{u,u} & A^{(m)}_{u,v} & A^{(m)}_{u,w}
\\
A^{(m)}_{v,\lambda} & A^{(m)}_{v,\alpha} & A^{(m)}_{v,u} & A^{(m)}_{v,v} & A^{(m)}_{v,w}
\\
A^{(m)}_{w,\lambda} & A^{(m)}_{w,\alpha} & A^{(m)}_{w,u} & A^{(m)}_{w,v} & A^{(m)}_{w,w}
\end{pmatrix}
\]
so that it acts on $x^{(m)} = (\lambda,\alpha,u^{(m)},v^{(m)},w^{(m)}) \in \C^{2m(n + 5) - 3}$.
Thus we define $A$ as
\begin{equation} \label{eq:A}
A =
\begin{pmatrix}
A_{\lambda,\lambda} & A_{\lambda,\alpha} & A_{\lambda,u} & A_{\lambda,v} & A_{\lambda,w}
\\
A_{\alpha,\lambda} & A_{\alpha,\alpha} & A_{\alpha,u} & A_{\alpha,v} & A_{\alpha,w}
\\
A_{u,\lambda} & A_{u,\alpha} & A_{u,u} & A_{u,v} & A_{u,w}
\\
A_{v,\lambda} & A_{v,\alpha} & A_{v,u} & A_{v,v} & A_{v,w}
\\
A_{w,\lambda} & A_{w,\alpha} & A_{w,u} & A_{w,v} & A_{w,w}
\end{pmatrix},
\end{equation}
where the action of each block of $A$ is finite (that is they act on $x^{(m)} = \Pi^{(m)} x$ only) except for the three diagonal blocks $A_{u,u}$, $A_{v,v}$ and $A_{w,w}$ which have infinite tails. More explicitly, for each $p=1,2,3$,
\begin{align*}
((A_{u,u} u )_p)_\ell &= 
\begin{cases}
\bigl( (A_{u,u}^{(m)} \pi_3^m u)_p \bigr)_\ell & \quad\text{for }   |\ell| < m, \\
\frac{1}{i \ell} (u_p)_\ell  & \quad\text{for }  |\ell| \ge m,
\end{cases}
\\
((A_{v,v} v )_p)_\ell &= 
\begin{cases}
\bigl( (A_{v,v}^{(m)} \pi_3^m v)_p \bigr)_\ell & \quad\text{for }   |\ell| < m, \\
\frac{1}{i \ell \omega^2} (v_p)_\ell  & \quad\text{for }  |\ell| \ge m,
\end{cases}
\end{align*}
and for each $j=1,\dots,n-1$, 
\[
((A_{w,w} w )_j)_\ell = 
\begin{cases}
\bigl( (A_{w,w}^{(m)} \pi_{n-1}^m w)_j \bigr)_\ell & \quad\text{for }   |\ell| < m, \\
\frac{1}{i \ell} (w_j)_\ell  & \quad\text{for }  |\ell| \ge m.
\end{cases}
\]

Having defined the operators $A$ and $\dagA$, we are ready to define the bounds $Y_0$, $Z_0$, $Z_1$ and $Z_2$ (satisfying \eqref{eq:Y0_radPolyBanach}, \eqref{eq:Z0_radPolyBanach}, \eqref{eq:Z1_radPolyBanach} and \eqref{eq:Z2_radPolyBanach}, respectively), required to build the radii polynomial defined on \eqref{eq:radii_polynomial}.

\section{The technical estimates for the Fourier map} \label{sec:bounds}

In this section, we introduce explicit formulas for the theoretical bounds \eqref{eq:Y0_radPolyBanach}, \eqref{eq:Z0_radPolyBanach}, \eqref{eq:Z1_radPolyBanach} and \eqref{eq:Z2_radPolyBanach}. While most of the work is analytical, the actual definition of the bounds still requires computing and verifying inequalities. In particular, there are many occasions in which the most practical means of obtaining necessary explicit inequalities is by using the computer. However, as  floating point arithmetic is only capable of representing a finite set of rational numbers, round off errors in the computation of the bounds can be dealt with by using interval arithmetic \cite{MR0231516} where real numbers are represented by intervals bounded by rational numbers that have floating point representation. Furthermore, there is software that performs interval arithmetic (e.g. {\tt INTLAB} \cite{Ru99a}) which we use for completing our computer-assisted proofs. With this in mind, in this section, when using phrases of the form {\em we can compute the following bounds}, this should be interpreted as shorthand for the statement {\em using the interval arithmetic software {\tt INTLAB} we can compute the following bounds}.

\subsection{\boldmath$Y_0$\unboldmath~bound} \label{sec:Y0}

Denote the numerical approximation $\bx = (\bar \lambda,\bar \alpha,\bu,\bv,\bw) \in X$ with 
$\bu=(\bu_1,\bu_2,\bu_3) \in (\ell_\nu^1)^3$, 
$\bv=(\bv_1,\bv_2,\bv_3) \in (\ell_\nu^1)^3$ and $\bw=(\bw_1,\dots,\bw_{n-1}) \in (\ell_\nu^1)^{n-1}$. 
Recalling \eqref{eq:f_Fourier_map_automatic_differentiation},  \eqref{eq:g_Fourier_map_automatic_differentiation}
and \eqref{eq:h_j_Fourier_map_automatic_differentiation}, one has that
\begin{align*}
(I - \iota_3^m \pi_3^m)f(\bu,\bv) &= 0 \in (\ell_\nu^1)^3
\\
(I - \iota_3^{4m-4} \pi_3^{4m-4}) g(\bar \lambda,\bu,\bv,\bw,\omega) &= 0 \in (\ell_\nu^1)^3
\\
(I - \iota_{n-1}^{5m-5} \pi_{n-1}^{5m-5}) h(\bar \alpha,\bu,\bv,\bw)&= 0 \in (\ell_\nu^1)^{n-1},
\end{align*}
since the product of $p$ trigonometric functions of degree $m-1$ is a trigonometric function of degree $p(m-1)$. For instance, recalling \eqref{eq:g_Fourier_map_automatic_differentiation}, the highest degree terms in $g(\bar \lambda,\bu,\bv,\bw,\omega)$ are of the form $(M_{j} \bu) * \bw_{j}^{3}$ which are convolutions of degree four, and therefore have zero Fourier coefficients for all frequencies $\ell$ such that $|\ell| > 4m-4$. This implies that $F(\bx,\omega)$ has only a finite number of nonzero terms. Hence, we can compute $Y_0$ satisfying \eqref{eq:Y0_radPolyBanach}.

\subsection{\boldmath$Z_0$\unboldmath~bound}

Let $B \bydef I - A A^{\dagger}$, which we denote block-wise by
\[
B =
\begin{pmatrix}
B_{\lambda,\lambda} & B_{\lambda,\alpha} & B_{\lambda,u} & B_{\lambda,v} & B_{\lambda,w}
\\
B_{\alpha,\lambda} & B_{\alpha,\alpha} & B_{\alpha,u} & B_{\alpha,v} & B_{\alpha,w}
\\
B_{u,\lambda} & B_{u,\alpha} & B_{u,u} & B_{u,v} & B_{u,w}
\\
B_{v,\lambda} & B_{v,\alpha} & B_{v,u} & B_{v,v} & B_{v,w}
\\
B_{w,\lambda} & B_{w,\alpha} & B_{w,u} & B_{w,v} & B_{w,w}
\end{pmatrix}.
\]
Note that by definition of the diagonal tails of $A$ and $\dagA$, the tails of $B$ vanish, that is all $B_{\delta,\tilde \delta}$ ($\delta,\tilde \in \{u,v,w\}$) are represented by $2m-1 \times 2m-1$ matrices. We can compute the bound
\[
{\tiny
Z_0^{(\delta)} \bydef
\left\{ 
\begin{array}{ll}
% \hspace{-.2cm}
\displaystyle
 \sum_{\tilde \delta \in \{ \lambda_1,\lambda_2,\lambda_3, \atop \alpha_1,\dots,\alpha_{n-1}\}} 
% \hspace{-.5cm}
 \left| B_{\delta,\tilde \delta} \right| 
 +  %\hspace{-.5cm}
  \sum_{\tilde \delta \in \{ u_1,u_2,u_3,v_1, \atop v_2,v_3,w_1,\dots,w_{n-1}\}} 
 % \hspace{-.7cm}
  \max_{|\ell|<m} 
  \frac{\left| \left( B_{\delta,\tilde \delta} \right)_\ell \right|}{\nu^{|\ell|}},
&
\delta \in \{ \lambda_1,\lambda_2,\lambda_3, \atop \alpha_1,\dots,\alpha_{n-1}\}  ,
\\
%\hspace{-.2cm}
\displaystyle
\sum_{\tilde \delta \in \{ \lambda_1,\lambda_2,\lambda_3, \atop \alpha_1,\dots,\alpha_{n-1}\}} 
\sum_{|\ell|<m}  \left| \left( B_{\delta,\tilde \delta} \right)_{\ell} \right| \nu^{|\ell|} 
+ %\hspace{-.5cm} 
\sum_{\tilde \delta \in \{ u_1,u_2,u_3,v_1, \atop  v_2,v_3,w_1,\dots,w_{n-1}\}} 
%\hspace{-.5cm}
\max_{|s|<m} \frac{1}{\nu^{|s|}} \sum_{|\ell|<m} \left| \left( B_{\delta,\tilde \delta} \right)_{\ell,s} \right| \nu^{|\ell|} ,& 
\delta \in \{ {u_1,u_2,u_3,\atop v_1,v_2,v_3,}\atop w_1,\dots,w_{n-1}\}.
\end{array}
\right.
}
\]
By construction, letting
\begin{equation} \label{eq:Z0_explicit}
Z_0 \bydef
\hspace{-.5cm}
 \max_{
{\delta \in \{ \lambda_1,\lambda_2,\lambda_3,}
\atop {\qquad \alpha_1,\dots,\alpha_{n-1}, \atop {\quad u_1,u_2,u_3, \atop {\quad v_1,v_2,v_3,\atop \qquad w_1,\dots,w_{n-1} \} }}}
} 
\hspace{-.5cm}
\left\{ Z_0^{(\delta)} \right\},
\end{equation}
we get that 
\[
\| I - A A^{\dagger}\|_{B(X)} \le Z_0.
\]

\subsection{\boldmath$Z_1$\unboldmath~bound}

Recall from \eqref{eq:Z1_radPolyBanach} that the $Z_1$ bound satisfy
\[
\| A[D_xF(\bx,\omega) - A^{\dagger} ] \|_{B(X)} \le Z_1.
\]
For the computation of this bound, it is convenient to define, given any $h \in B_1(0) \in X$
\begin{equation} \label{eq:small_z(h)}
z = z(h) \bydef [D_xF(\bx,\omega) - A^{\dagger} ]h.
\end{equation}
Denote 
\begin{align*}
h &= (h_\lambda,h_\alpha,h_u,h_v,h_w) \in \C^3 \times \C^{n-1} \times (\ell_\nu^1)^3 \times (\ell_\nu^1)^3 \times (\ell_\nu^1)^{n-1},
\\
z &= (z_\lambda,z_\alpha,z_u,z_v,z_w) \in \C^3 \times \C^{n-1} \times (\tilde \ell_\nu^1)^3 \times (\tilde \ell_\nu^1)^3 \times (\tilde \ell_\nu^1)^{n-1}.
\end{align*}

The construction of $Z_1$ hence requires computing an upper bound for $\|A z\|_X$ for all $h \in B_1(0) \in X$. 
This is done by splitting $Az$ as 
\begin{align}
\label{eq:Az_splitting}
Az &= \inc \Pi^{(m)} A z + (I-\inc \Pi^{(m)} )A z 
\\
&= \inc A^{(m)} z^{(m)} + 
\begin{pmatrix} 0 \\ 0 \\ 
(I - \iota_3^m \pi_3^m) D^{-1} z_u
\\
\frac{1}{\omega^2} (I - \iota_3^m \pi_3^m) D^{-1} z_v
\\
(I - \iota_{n-1}^m \pi_{n-1}^m) D^{-1} z_w
\end{pmatrix}
\nonumber
\end{align}
and by handling each term separately. 

\begin{remark}
We choose the Galerkin projection number $m$ greater than the number $m_1$ of nonzero Fourier coefficients of the previous orbit $(\tu_1,\tu_2,\tu_3)$.  Then $z_\lambda =0 \in \C^3$. This is because the phase conditions $\eta(u)$ defined in \eqref{eq:eta} only depend on the modes of the finite dimensional approximation and therefore $A^\dagger$ contains all contribution from $D\eta(\bu) h$.
\end{remark}

As $\Pi^{(m)} A z = A^{(m)} z^{(m)}$, we compute a uniform component-wise upper bound
\[
\hat z^{(m)} = (0,\hat z_\alpha,\hat z^{(m)}_u,\hat z^{(m)}_v,\hat z^{(m)}_w) \in \R_+^{2m(n + 5) - 3}
\]
for the complex modulus of each component of
\[
\Pi^{(m)} z  = z^{(m)} = (0,z_\alpha,z^{(m)}_u,z^{(m)}_v,z^{(m)}_w) \in \C^{2m(n + 5)-3}.
\]
%
%that is 
%
%\[
%\left| z_k^{(m)} \right|  = \left| z_k^{(m)}(h) \right| \le \hat z_k^{(m)}, \quad \text{for all } h \in B_1(0).
%\]
%
The computation of the bounds $\hat z_\alpha$, $\hat z^{(m)}_u$, $\hat z^{(m)}_v$ and $\hat z^{(m)}_w$ is done in Sections~\ref{sec:hzalpha},~\ref{sec:hzu},~\ref{sec:hzv} and \ref{sec:hzw}, respectively. Using these uniform bounds (i.e. for all $h \in B_1(0)$), let 
\begin{equation} \label{eq:xi_uniform_bounds}
\xi^{(m)}= \left( 
\xi_\lambda^{(m)}, \xi_\alpha^{(m)}, \xi_u^{(m)}, \xi_v^{(m)}, \xi_w^{(m)}
\right)  \bydef |A^{(m)}| \hat z^{(m)}   \in \R_+^{2m(n + 5) - 3},
\end{equation}
where the entries of the matrix $|A^{(m)}|$ are the component-wise complex magnitudes of the entries of $A^{(m)}$.
By construction, the bound $\xi^{(m)}$ of \eqref{eq:xi_uniform_bounds} provides a uniform component-wise upper bound for the first term $\inc \Pi^{(m)} A z$ of the splitting \eqref{eq:Az_splitting} of $Az$. To handle the second term $(I-\inc \Pi^{(m)} )A z$ of \eqref{eq:Az_splitting}, we compute the uniform (i.e. for all $h \in B_1(0)$) tail bounds $(\delta_u)_p$, $(\delta_v)_p$ (for $p=1,2,3$) and $(\delta_w)_j$ (for $j=1,\dots,n-1$) satisfying
\begin{align*}
\sum_{|\ell| \ge m} \left| \frac{1}{i \ell} ((z_u)_p)_\ell \right| \nu^{|\ell|} & \le (\delta_u)_p,\quad p=1,2,3
\\
\sum_{|\ell| \ge m} \left| \frac{1}{i \ell \omega^2} ((z_v)_p)_\ell \right| \nu^{|\ell|} & \le (\delta_v)_p,\quad p=1,2,3
\\
\sum_{|\ell| \ge m} \left| \frac{1}{i \ell} ((z_w)_j)_\ell \right| \nu^{|\ell|} & \le (\delta_w)_j,\quad j=1,\dots,n-1.
\end{align*}
The computation of the bounds $\delta_u$, $\delta_v$ and $\delta_w$ is presented in Sections~\ref{sec:hzu}, \ref{sec:hzv} and \ref{sec:hzw}, respectively. Combining the above bounds, we get that
\begin{align} \label{eq:Z1_explicit}
\|A z\|_X &\le Z_1 \bydef \max \left\{ |\xi_\lambda^{(m)}|_\infty, |\xi_\alpha^{(m)}|_\infty, \max_{p=1,2,3} \left( \| \iota^m(\xi_u^{(m)})_p \|_{\nu} + (\delta_u)_p \right), \right. \\
& \hspace{2.5cm} \max_{p=1,2,3} \left( \| \iota^m(\xi_v^{(m)})_p \|_{\nu} + (\delta_v)_p \right), \left. 
\max_{j=1,\dots,n-1} \left( \| \iota^m(\xi_w^{(m)})_j \|_{\nu} + (\delta_w)_j \right)
\right\}.
\nonumber
\end{align}

\subsubsection{Computation of the bound \boldmath$\hat z_\alpha$\unboldmath} \label{sec:hzalpha}

Recalling \eqref{eq:small_z(h)} and the definition of $\dagA$ in \eqref{eq:A_dagger}, one can verify that for any $j=1,\dots,n-1$,
\begin{align*}
(z_{\alpha})_j &= 
2 \left(  \sum_{\ell\in\mathbb{Z}} (\bw_{j})_{\ell}\right) \left[
\sum_{p=1}^{3} \left(  \sum_{\ell\in\mathbb{Z}} (M_{j \ell}\bu_{\ell})_{p}
\right) ^{2} \right] \left(  \sum_{|\ell| \ge m } ((h_u)_{j})_{\ell}\right) \\
& \quad + 2 \left(  \sum_{\ell\in\mathbb{Z}} (\bw_{j})_{\ell}\right) ^{2} \left[
\sum_{p=1}^{3} \left(  \sum_{\ell\in\mathbb{Z}} (M_{j \ell} \bu_{\ell})_{p}
\right) \left(  \sum_{|\ell| \ge m} (M_{j \ell} (h_u)_{\ell})_{p}
\right) \right].
\end{align*}
Straightforward calculations (e.g. using Lemma 2.1 in \cite{LeMi16}) involving bounding linear functionals on $\ell_\nu^1$ and using that $(h_u)_p \in B_1(0) \subset \ell_\nu^1$ for $p=1,2,3$ yield that 
\[
\left|  \sum_{|\ell| \ge m } ((h_u)_{j})_{\ell} \right| \le \frac{1}{\nu^m}, 
\quad 
\left|  \sum_{|\ell| \ge m} (M_{j \ell} (h_u)_{\ell})_{p} \right| \le \frac{i_p}{\nu^m}, 
\quad 
i_p \bydef \begin{cases} 3, & p=1,2 \\ 2, & p=3. \end{cases}
\]
We therefore get the component-wise bound  (given $j=1,\dots,n-1$)
\begin{align} \label{eq:hat_z_alpha}
|(z_{\alpha})_j|  &\le (\hat z_{\alpha})_j 
\\
& \bydef
\frac{2}{\nu^m} \left[
\left( \sum_{\ell \in \Z } (\bw_{j})_{\ell} \right)
\sum_{p=1}^{3} \left( \sum_{\ell \in \Z} (M_{j \ell} \bu_{\ell} )_{p} \right)^{2} 
+ \left( \sum_{\ell \in \Z} (\bw_{j})_{\ell} \right)^2 
\sum_{p=1}^{3} \left( \sum_{\ell \in \Z} (M_{j \ell} \bu_{\ell})_{p} \right) i_p \right].
\nonumber
\end{align}

\subsubsection{Computation of the bounds \boldmath$\hat z^{(m)}_u$\unboldmath~and~\boldmath$\delta_u$\unboldmath} \label{sec:hzu}

From \eqref{eq:small_z(h)} and \eqref{eq:A_dagger}, on can verify that for each $p=1,2,3$, 
\[
((z_u)_p)_\ell = 
\begin{cases}
0, & |\ell| <m \\ -((h_u)_p)_\ell, & |\ell| \ge m.
\end{cases}
\]
Hence, since $z_u$ only has a tail and since the blocks $A_{\lambda,u}$, $A_{\alpha,u}$, $A_{v,u}$ and $A_{w,u}$ only acts on the finite part, then $A_{\delta,u} z_u = 0$ for $\delta = \lambda,\alpha,v,w$ and for $p=1,2,3$
\[
((A_{u,u} z_u)_p)_\ell  = -\frac{1}{i\ell}((h_u)_p)_\ell.
\]

Now, 
\[
\sum_{|\ell| \ge m} \left| -\frac{1}{i\ell}((h_u)_p)_\ell \right| \nu^{|\ell|} \le \frac{1}{m}
\sum_{|\ell| \ge m} \left| ((h_u)_p)_\ell \right| \nu^{|\ell|} \le \frac{1}{m} \| (h_u)_p\|_\nu \le \frac{1}{m}.
\]

We can then set 
\begin{align}
\label{eq:hzu}
\hat z^{(m)}_u & \bydef 0 \in \R^{3(2m-1)} 
\\
\label{eq:delta_u}
(\delta_u)_p & \bydef \frac{1}{m}, \quad p=1,2,3.
\end{align}

\subsubsection{Computation of the bound \boldmath$\hat z^{(m)}_v$\unboldmath~and~\boldmath$\delta_v$\unboldmath} \label{sec:hzv}

The following technical lemma (which is a slight modification of Corollary 3 in \cite{LeMi16}) is the key to the truncation error analysis of $\hat z_v^{(m)}$ and $\hat z_w^{(m)}$.
\begin{lemma}\label{lem:Z1_estimate_ps}
Fix a truncation Fourier mode to be $m$. Given $h \in \ell_\nu^1$, set 
\[
h^{(I)} \bydef (I- \iota^m \pi^m ) h = (\ldots, h_{-m-1}, h_{-m}, 0, \ldots, 0, h_{m}, h_{m+1}, \ldots) \in \ell_\nu^1.
\]
Let $N \in \N$ and let $\bar \alpha = (\ldots, 0, 0, \bar \alpha_{-N}, \ldots, \bar \alpha_{N}, 0, 0, \ldots) \in \ell_\nu^1$. Then, for all $h \in \ell_\nu^1$ such that $\|h\|_{\nu} \le 1$, and for $|\ell|<m$,
\begin{equation} \label{eq:bounds_lkac}
\left| (\bar \alpha * h^{(I)})_\ell \right| \le 
\Psi_\ell(\bar \alpha) \bydef
\max \left(
\max_{\ell-N \le s \le -m}  \frac{|\bar \alpha_{\ell- s}|}{\nu^{|s|}} ,
\max_{m \le s \le \ell+N} \frac{| \bar \alpha_{\ell - s}|}{\nu^{|s|}} 
\right).
\end{equation}
\end{lemma} 

Now, from \eqref{eq:small_z(h)} and \eqref{eq:A_dagger}, on can verify that for each $p=1,2,3$, 
\[
((z_v)_p)_\ell = 
\begin{cases}
\left( 
\displaystyle
\sum_{j=1}^{n-1} (M_j h_u^{(I)})_p*\bw_j^3 + 3 (M_j \bu )_p*\bw_j^2*h_{w_j}^{(I)}
 \right)_\ell, & |\ell| <m 
\\ 
\left( \left(
2 \omega \sqrt{s_1} \bar J h_v - s_1 \bar I h_u  + (h_\lambda)_1 \bar J \bu
+ \bar \lambda_1 \bar J h_u + \bar \lambda_2 h_v + (h_\lambda)_2 \bv \right)_p \right)_\ell  &
\\
+ \left(
\displaystyle
% (h_\lambda)_3 \hat e_3 + 
\sum_{j=1}^{n-1} (M_j h_u)_p*\bw_j^3 + 3 (M_j \bu )_p*\bw_j^2*h_{w_j}
\right)_\ell, & |\ell| \ge m.
\end{cases}
\]
Using Lemma~\ref{lem:Z1_estimate_ps}, we obtain that for $|\ell| <m$ and $p=1,2,3$,
\begin{equation} \label{eq:hzv}
\left| ((z^{(m)}_v)_p)_\ell \right| \le \left( (\hat z^{(m)}_v )_p\right)_\ell \bydef \sum_{j=1}^{n-1} i_p \Psi_\ell(\bw_j^3) + 3 \Psi_\ell( (M_j \bu)_p*\bw_j^2),
\end{equation}
which provides a component-wise definition of the vector $\hat z_v^{(m)} \in \R_+^{3(2m-1)}$. 
Finally, one can verify using the fact that $\ell_\nu^1$ is a Banach algebra, that
\begin{align}
\nonumber
\sum_{|\ell| \ge m} \left| \frac{1}{i \ell \omega^2} ((z_v)_1)_\ell \right| \nu^{|\ell|}  \le
(\delta_v)_1 & \bydef 
\frac{1}{m \omega^2} \big( 
2\omega \sqrt{s_1} + s_1 + \|\bu_2\|_\nu + |\bar \lambda_1|  + |\bar \lambda_2| + \|\bv_1\|_\nu
\\
\label{eq:delta_v1}
& \hspace{1.5cm} + 3 \sum_{j=1}^{n-1} \|\bw_j\|_\nu^3 + \| (M_j \bu)_1\|_\nu \|\bw_j\|_\nu^2 \big)
\\
\nonumber
\sum_{|\ell| \ge m} \left| \frac{1}{i \ell \omega^2} ((z_v)_2)_\ell \right| \nu^{|\ell|}  \le
(\delta_v)_2 & \bydef 
\frac{1}{m \omega^2} \big( 
2\omega \sqrt{s_1} + s_1 + \|\bu_1\|_\nu + |\bar \lambda_1|  + |\bar \lambda_2| + \|\bv_2\|_\nu
\\
\label{eq:delta_v2}
& \hspace{1.5cm} + 3 \sum_{j=1}^{n-1} \|\bw_j\|_\nu^3 + \| (M_j \bu)_2\|_\nu \|\bw_j\|_\nu^2 \big)
\\
\label{eq:delta_v3}
\sum_{|\ell| \ge m} \left| \frac{1}{i \ell \omega^2} ((z_v)_3)_\ell \right| \nu^{|\ell|}  \le
(\delta_v)_3 & \bydef 
\frac{1}{m \omega^2} \big( 
 |\bar \lambda_2| + \|\bv_3\|_\nu
 + \sum_{j=1}^{n-1} 2 \| \bw_j\|_\nu^3 + 3 \| (M_j \bu)_3\|_\nu \|\bw_j\|_\nu^2 \big)
 \end{align}

\subsubsection{Computation of the bound \boldmath$\hat z^{(m)}_w$\unboldmath~and~\boldmath$\delta_w$\unboldmath} \label{sec:hzw}

From \eqref{eq:small_z(h)} and \eqref{eq:A_dagger}, on can verify that for each $j=1,\dots,n-1$,
\[
%\hspace{-3.3cm}
((z_w)_j)_\ell = 
{\tiny
\begin{cases}
\left( 3 \bar \alpha_j \bw_j^2 * h^{(I)}_{w_j} \right)_\ell  & \\
\displaystyle
+ \left(
\sum_{p=1}^3 3 \bw_j^2 * h_{w_j}^{(I)} * (M_j \bu )_p * (M_j \bv)_p
+ \bw_j^3 * \left( (M_j h_u^{(I)} )_p * (M_j \bv)_p + (M_j \bu )_p * (M_j h_v^{(I)})_p \right) 
\right)_\ell, & |\ell| <m 
\\ 
\left(
h_{\alpha_j} \bw_j^3 + 3 \bar \alpha_j \bw_j^2 * h_{w_j}
\right)_\ell  &
\\
+ \left( \displaystyle
\sum_{p=1}^3 3 \bw_j^2 * h_{w_j} * (M_j \bu )_p * (M_j \bv)_p
+ \bw_j^3 * \left( (M_j h_u )_p * (M_j \bv)_p + (M_j \bu )_p * (M_j h_v)_p \right) 
 \right)_\ell, & |\ell| \ge m.
\end{cases}
}
\]

Using Lemma~\ref{lem:Z1_estimate_ps}, we obtain that for $|\ell| <m$ and $j=1,\dots,n-1$,
\begin{align} 
\label{eq:hzw}
\left| ((z^{(m)}_w)_j)_\ell \right| \le \left( (\hat z^{(m)}_w )_p\right)_\ell 
& \bydef 3 |\bar \alpha_j| \Psi_\ell(\bw_j^2) +  \sum_{p=1}^3 3 \Psi_\ell(\bw_j^2*(M_j \bu )_p*(M_j \bv )_p) 
\\
& \qquad + \sum_{p=1}^3 i_p \Psi_\ell( \bw_j^3*(M_j \bv)_p) + i_p \Psi_\ell( \bw_j^3*(M_j \bu)_p).
\nonumber
\end{align}
Moreover, for $j=1,\dots,n-1$,
\begin{align}
\nonumber
\sum_{|\ell| \ge m} \left| \frac{1}{i \ell} (((z_w)_j)_\ell \right| \nu^{|\ell|} \le
(\delta_w)_j &
\hspace{-.1cm}
\bydef 
\hspace{-.1cm}
\frac{1}{m} \hspace{-.1cm} \left( 
\| \bw_j\|_\nu^3 + 3 |\bar \alpha_j| \|\bw_j\|_\nu^2 + \sum_{p=1}^{3} 3 \|\bw_j\|_\nu^2 \| (M_j \bu)_p\|_\nu \| (M_j \bv)_p \|_\nu
\right.
\\
\label{eq:delta_wj}
& \hspace{1.3cm}
\left. 
+ \sum_{p=1}^{3} i_p \|\bw_j\|_\nu^3 (\| (M_j \bu)_p\|_\nu+\| (M_j \bv)_p\|_\nu)
 \right).
\end{align}

Combining \eqref{eq:hat_z_alpha}, \eqref{eq:hzu}, \eqref{eq:hzv} and \eqref{eq:hzw}, we define the uniform bound $\hat z^{(m)}$ which is then used to compute $\xi^{(m)}$ in \eqref{eq:xi_uniform_bounds}. Moreover, combining \eqref{eq:delta_u}, \eqref{eq:delta_v1}, \eqref{eq:delta_v2}, \eqref{eq:delta_v3} and \eqref{eq:delta_wj} provides the explicit bounds $\delta_u$, $\delta_v$ and $\delta_w$. All of these uniform bounds combined are finally used to compute the bound $Z_1$ in \eqref{eq:Z1_explicit} which by construction satisfy \eqref{eq:Z1_radPolyBanach}. 

\subsection{\boldmath$Z_2$\unboldmath~bound}

Recall that we look for a bound $Z_2$ satisfying \eqref{eq:Z2_radPolyBanach}. Consider $Z_2$ satisfying 
\[
\|A\|_{B(X)} \sup_{{\xi \in B_r(\bx) \atop h^{(1)},h^{(2)} \in B_1(0)}} 
\| D^2_xF(\xi,\omega)(h^{(1)},h^{(2)})\|_{X} \le Z_2.
\]
Then, for any $b \in B_r(0)$, applying the Mean Value Inequality yields
\begin{align*}
\| A [D_x F(\bx+b,\omega) - D_xF(\bx,\omega)]\|_{B(X)} 
& \le r \sup_{{\xi \in B_r(\bx) \atop h^{(1)},h^{(2)} \in B_1(0)}} 
\| A D^2_xF(\xi,\omega)(h^{(1)},h^{(2)})\|_{X}  \le Z_2 r.
\end{align*}
Given $\xi \in B_r(\bx)$ and $h^{(1)},h^{(2)} \in B_1(0)$, we aim at bounding $\| D^2_xF(\xi,\omega)(h^{(1)},h^{(2)})\|_{X}$. Let
\[
z \bydef D^2_xF(\xi,\omega)(h^{(1)},h^{(2)}),
\]
which we denote by $z=(z_\lambda,z_\alpha,z_u,z_v,z_w) = (0,z_\alpha,0,z_v,z_w)$, where $z_\lambda$ and $z_u$ are both zero since $\eta$ and $f$ are linear. Denote
\begin{align*}
h^{(i)} &= \left( h_\lambda^{(i)} , h_\alpha^{(i)} , h_u^{(i)} , h_v^{(i)} , h_w^{(i)} \right), \quad i=1,2
\\
\xi &= \left( \xi_\lambda , \xi_\alpha , \xi_u , \xi_v , \xi_w \right).
\end{align*}
Then, for $j=1,\dots,n-1$,
\begin{align*}
(z_\alpha)_j & = 2 \left[ \sum_{\ell \in \Z}  \left( h_{w_j}^{(2)} \right)_\ell \right] \left[ \sum_{\ell \in \Z} \left( h_{w_j}^{(1)} \right)_\ell \right] \sum_{p=1}^3 \left( \sum_{\ell \in \Z} (M_{j \ell} (\xi_u)_\ell )_p \right)^2 \\
& \qquad + 4 \left[ \sum_{\ell \in \Z}  \left( \xi_{w_j} \right)_\ell \right] \left[ \sum_{\ell \in \Z} \left( h_{w_j}^{(1)} \right)_\ell \right] \sum_{p=1}^3 \left( \sum_{\ell \in \Z} (M_{j \ell} (\xi_u)_\ell )_p \right) \left( \sum_{\ell \in \Z} (M_{j \ell} (h_u^{(2)})_\ell )_p \right) \\
& \qquad + 4 \left[ \sum_{\ell \in \Z}  \left( \xi_{w_j} \right)_\ell \right] \left[ \sum_{\ell \in \Z} \left( h_{w_j}^{(2)} \right)_\ell \right] \sum_{p=1}^3 \left( \sum_{\ell \in \Z} (M_{j \ell} (\xi_u)_\ell )_p \right) \left( \sum_{\ell \in \Z} (M_{j \ell} (h_u^{(1)})_\ell )_p \right) \\
& \qquad + 2 \left( \sum_{\ell \in \Z}  \left( \xi_{w_j} \right)_\ell \right)^2 
\sum_{p=1}^3 \left( \sum_{\ell \in \Z} (M_{j \ell} (h_u^{(2)})_\ell )_p \right) \left( \sum_{\ell \in \Z} (M_{j \ell} (h_u^{(1)})_\ell )_p \right).
\end{align*}
Consider $r_*>0$ such that $r \le r_*$.
For $j=1,\dots,n-1$ and $i=1,2$,
\begin{align*}
\left| \sum_{\ell \in \Z}  \left( h_{w_j}^{(i)} \right)_\ell \right| & \le  \sum_{\ell \in \Z}  \left| \left( h_{w_j}^{(i)} \right)_\ell \right| \nu^{|\ell|} = \| h_{w_j}^{(i)} \|_\nu \le 1
\\
\left| \sum_{\ell \in \Z}  \left( \xi_{w_j} \right)_\ell \right| & \le \|  \xi_{w_j} \|_\nu
\le \|  \bw_j \|_\nu  + r \le \hat w_j \bydef 
\|  \bw_j \|_\nu + r_* 
\\
\left| \sum_{\ell \in \Z} (M_{j \ell} (\xi_u)_\ell )_p \right| & \le 
\hat \delta_p(u) \bydef
\begin{cases}
2 \| \bu_1\|_\nu + \| \bu_2\|_\nu + 3r_* , & p=1 
\\
\| \bu_1\|_\nu + 2 \| \bu_2\|_\nu  + 3r_* , & p=2
\\
2 \| \bu_3 \|_\nu  + 2r_* , & p=3
\end{cases}
\\
\left| \sum_{\ell \in \Z} (M_{j \ell} (h_u^{(i)})_\ell )_p \right| & \le i_p.
\end{align*}
Then, for $j=1,\dots,n-1$,
\begin{equation} \label{eq:hat_z_alpha_j}
\left| (z_\alpha)_j \right| \le (\hat z_\alpha)_j \bydef 2 \sum_{p=1}^3 \hat \delta_p(u)^2 + 4 \hat w_j i_p \hat \delta_p(u) + \hat w_j^2 i_p^2.
\end{equation}
One verifies that 
\begin{align*}
z_v &= h_{\lambda_1}^{(1)} \bar J h_{u}^{(2)} + h_{\lambda_1}^{(2)} \bar J h_{u}^{(1)} +
h_{\lambda_2}^{(1)} h_{v}^{(2)} + h_{\lambda_2}^{(2)} h_{v}^{(1)} 
\\
& \quad + 3\sum_{j=1}^{n-1} (M_j h_u^{(1)})*(\xi_{w_j})^2*h_{w_j}^{(2)} + (M_j h_u^{(2)})*(\xi_{w_j})^2*h_{w_j}^{(1)} 
+ 2 (M_j \xi_u)*\xi_{w_j}*h_{w_j}^{(2)}*h_{w_j}^{(1)},
\end{align*}
and hence using the Banach algebra structure of $\ell_\nu^1$, we get that (for $p=1,2,3$)
\begin{equation} \label{eq:hat_v_p}
\| (z_v)_p \|_\nu \le (\hat z_v)_p \bydef 
4 + 6 \sum_{j=1}^{n-1} i_p \hat w_j^2 + \hat \delta_p(u) \hat w_j.
\end{equation}
For $j=1,\dots,n-1$,
\begin{align*}
z_{w_j} & = 6 \xi_{w_j} * h_{w_j}^{(2)}* h_{w_j}^{(1)} * \sum_{p=1}^3 (M_j \xi_u)_p * (M_j \xi_v)_p 
\\
& \quad  
+ 3 (\xi_{w_j})^2 * h_{w_j}^{(1)} * \sum_{p=1}^3 \left( (M_j h_u^{(2)})_p * (M_j \xi_v)_p + (M_j \xi_u)_p * (M_j h_v^{(2)})_p  \right)
\\
& \quad  
+ 3 (\xi_{w_j})^2 * h_{w_j}^{(2)} * \sum_{p=1}^3 \left( (M_j h_u^{(1)})_p * (M_j \xi_v)_p + (M_j \xi_u)_p * (M_j h_v^{(1)})_p  \right)
\\
& \quad  
+ (\xi_{w_j})^3 * \sum_{p=1}^3 \left( (M_j h_u^{(1)})_p * (M_j h_v^{(2)})_p + (M_j h_u^{(2)})_p * (M_j h_v^{(1)})_p  \right) 
\\
& \quad
+ 3 h_{\alpha_j}^{(1)} (\xi_{w_j})^2 * h_{w_j}^{(2)} + 3 h_{\alpha_j}^{(2)} (\xi_{w_j})^2 * h_{w_j}^{(1)} 
+ 6 \xi_{\alpha_j} \xi_{w_j} * h_{w_j}^{(2)}* h_{w_j}^{(1)},
\end{align*}
and hence,
\begin{align}
\nonumber
\| z_{w_j} \|_\nu \le \hat z_{w_j} &
\bydef 2 \hat w_j  \sum_{p=1}^3 \left( 3\hat \delta_p(u) \hat \delta_p(v) + 3 \hat w_j i_p ( \hat \delta_p(u) + \hat \delta_p(v) ) + \hat w_j i_p^2 \right)
\\
& \quad + 6 \hat w_j (\hat w_j + |\bar \alpha_j| + r_*).
\label{eq:hat_z_w_j}
\end{align}

Combining \eqref{eq:hat_z_alpha_j}, \eqref{eq:hat_v_p} and \eqref{eq:hat_z_w_j}, set
\begin{equation} \label{eq:Z2_explicit}
Z_2 \bydef \|A\|_{B(X)} \max_{j=1,\dots,n-1 \atop p=1,2,3} \{ (\hat z_\alpha)_j , (\hat z_v)_p, \hat z_{w_j} \}
\end{equation}
and therefore, for all $b\in B_r(0)$,
\[
\| A [D_xF(\bx+b,\omega) - D_xF(\bx,\omega)]\|_{B(X)} \le Z_2 r. 
\]

\section{Results} \label{sec:results}

In this section, we present several computer-assisted proofs of existence of spatial torus-knot choreographies. First fix the number of bodies $n$, a prescribed symmetry \eqref{Sy} (determined by the integer $k$), a resonance $(p,q)$, the frequency $\omega$ given in \eqref{di}, and a Galerkin projection number $m$. Then compute a {\em real} numerical approximation $\bx \in X_\real$ of the finite dimensional projection $F^{(m)}$ defined in \eqref{eq:F_Fourier_map_projection}, where $X_\real$ is defined in \eqref{eq:sym_Banach_space}. Define the operators $\dagA$ and $A$ as in Section~\ref{sec:operators_A_dagA}. Since the {\em tail} of the diagonal blocks of the approximate inverse $A$ (which is defined in \eqref{eq:A}) involves the operator $D^{-1}$, we can easily show (using that $\ell_\nu^1$ is a Banach algebra under discrete convolutions) that the hypothesis \eqref{eq:AF:X->X} of Theorem~\ref{thm:radPolyBanach} holds, that is $A F \colon X \times \R \to X$. Having described how to compute the bounds $Y_0$ in Section~\ref{sec:Y0}, $Z_0$ in \eqref{eq:Z0_explicit}, $Z_1$ in \eqref{eq:Z1_explicit} and $Z_2$ in \eqref{eq:Z2_explicit}, we have all the ingredients to compute the radii polynomial defined in \eqref{eq:radii_polynomial}. The proof of existence then reduces to verify rigorously the hypothesis \eqref{eq:p(r0)<0} of Theorem~\ref{thm:radPolyBanach}. This is done with a computer program in {\tt MATLAB} implemented with the interval arithmetic package {\tt INTLAB}, and available at \cite{webpage}. 
All computations are performed with 16 decimal digits' precision.

 Let us present in details the computer-assisted proof resulting in the constructive existence of the torus-knot choreography of Figure~\ref{fig:trefoil}. 

\begin{theorem}  \label{thm:trefoil}
Fix $n=5$ and consider the symmetry \eqref{Sy} with $k=3$. Let $(p,q)=(3,1)$ be the resonance. Let $s_1=\frac{1}{4}\sum_{j=1}^{4}\frac{1}{\sin(j\pi/5)}$ be given by \eqref{s1} and the frequency $\omega=3\sqrt{s_{1}}$ be as in \eqref{di}. Fix the Galerkin projection number $m=25$ and the decay rate parameter $\nu=1.03$. Consider the numerical approximation 
\[
\bu(t)  = \sum_{|\ell|<25} \left( (\bu_1)_{\ell} , (\bu_2)_{\ell} , (\bu_3)_{\ell} \right) e^{i \ell t}, \quad (\bu_j)_{-\ell}  = \left( (\bu_j)_{\ell} \right)^*,
\]
where the real and the imaginary part of the Fourier coefficients $(\bu_j)_{\ell}$ can be found in the Appendix in Table~\ref{table:trefoil_data}. Then there exist sequences $\tu_1,\tu_2,\tu_3 \in \ell_\nu^1$ such that 
\begin{equation} \label{eq:tu_thm}
\tu(t) \bydef \sum_{\ell \in \Z} \left( (\tu_1)_{\ell} , (\tu_2)_{\ell} , (\tu_3)_{\ell} \right) e^{i \ell t}, \quad (\tu_j)_{-\ell}  = \left( (\tu_j)_{\ell} \right)^*
\end{equation}
with
\[
\| \bu_j - \tu_j\|_{C^2}  \le 4.7 \times 10^{-10}, \qquad \text{for each } j=1,2,3,
\]
and such that $\mathcal{G}(\tu,\omega)=0$, with $\mathcal{G}$ defined in \eqref{eq:DDE_formulation}. Then $\{Q_j\}_{j=1}^5$ defined in the inertial frame by
\begin{align} 
\label{eq:tQ5_proof}
 Q_{5}(t) \bydef e^{t\bar{J}/3} \tilde u(t)
,  \quad Q_{j}(t) \bydef  Q_{5}(t+3j\zeta), \quad j=1,2,3,4
\end{align}
%
%By Proposition~\ref{proposition}, the $2\pi$-periodic solution $\tu$ defined in \eqref{eq:tu_thm} 
is a (renormalized) $6\pi $-periodic choreography that is symmetric by  $2\pi/3 $-rotations. Moreover, there exist countably many choreographies with frequencies near $\omega=3\sqrt{s_{1}}$.
\end{theorem}

\begin{proof}
First denote by $\bx = (\bar \lambda, \bar \alpha,\bu,\bv,\bw) \in \mathbb{C}^{2m(n + 5) - 3} = \C^{497}$ a numerical approximation of the finite dimensional reduction $F^{(497)}:\C^{497} \to \C^{497}$ defined in \eqref{eq:F_Fourier_map_projection}. The approximation satisfies $\bx \in X_\real$ and can be found in the file {\tt pt\_five\_bodies.mat} available at \cite{webpage}. Note that $\bu \in \C^{3(2m-1)} = \C^{147}$ is recovered from the coefficients in Table~\ref{table:trefoil_data} of the Appendix. Fix $\nu=1.03$. The {\tt MATLAB} computer program {\tt proof\_five\_bodies.m} available at \cite{webpage} computes $Y_0$ as in Section~\ref{sec:Y0}, $Z_0$ in \eqref{eq:Z0_explicit}, $Z_1$ in \eqref{eq:Z1_explicit} and $Z_2$ in \eqref{eq:Z2_explicit}, and verifies rigorously (using {\tt INTLAB}) the hypothesis \eqref{eq:p(r0)<0} of Theorem~\ref{thm:radPolyBanach} with $r_0=4.7 \times 10^{-10}$. Combining Theorem~\ref{thm:radPolyBanach} and Proposition~\ref{prop:real_solutions}, there exists
$\tx = (\tilde \lambda, \tilde \alpha,\tu,\tilde v,\tilde w) \in X_\real$ such that $F(\tx,\omega)=0$ and $\| \tx - \bx\|_X \le r_0 = 4.7 \times 10^{-10}$. Hence, for a given $j \in \{1,2,3\}$,
\begin{align*}
\| \tu_j - \bu_j\|_{C^2} = \| \tu_j - \bu_j \|_\nu \le \| \tx - \bx\|_X \le r_0 = 4.7 \times 10^{-10}.
\end{align*}
By construction of the Fourier map $F$ introduced in Section~\ref{sec:Fourier_map}, the solution $\tx$ yields a $2\pi$-periodic solution $(\tu,\tilde v,\tilde w)$ of the delay equations \eqref{eq:dde1_automatic_differentiation}, \eqref{eq:dde2_automatic_differentiation} and \eqref{eq:dde3_automatic_differentiation}, which also satisfies the extra condition \eqref{eq:extra_initial_conditions}. By Proposition~\ref{prop:alpha=0}, $\tu$ satisfies 
$\mathcal{G}(\tu,\omega)=0$.  The result follows from Proposition~\ref{proposition}. The existence of countably many choreographies with frequencies near $\omega=3\sqrt{s_{1}}$ follows from Corollary~\ref{cor:nonDegen} and the discussion thereafter. 
\end{proof}

In the two left subfigures of Figure~\ref{fig:fiveBod}, we can visualize (in red) the $2\pi$-periodic solution $\tu$ satisfying the reduced delay equations \eqref{eq:reduced_equations}. The initial condition $\tu(0) = (x_0,y_0,z_0)$ of that red orbit can be found in Table~\ref{table:initial_conditions} of the Appendix. This orbit is in the rotating frame. Still in the rotating frame, the position of the other bodies (in blue) can be recovered via the symmetry \eqref{Sy}. In the two right subfigures of Figure~\ref{fig:fiveBod}, we can visualize the position of the bodies $ Q_1, \dots, Q_5$, which are now in the inertial frame. Since 3 and  5  are relative prime, the factor 3 in the equality $
Q_{j}(t)=Q_{n}(t+3j\zeta )$  is just a re-ordering of the numbering of the bodies $j=1,2,3,4$.

\begin{remark}[Resonance numbers versus torus winding] \label{rem:versus}
{\em When $u_{n}(t)$ is a $p:q$ resonant orbit in the axial family with zero 
winding with respect to the $z$-axis, the choreography $Q_{n}(t)$ is a $(p,q)$%
-torus knot.  See Corollary \ref{cor:torusKnots}.
In this case the resonance order $p:q$ in our functional analytic set-up corresponds
exactly to the windings in the definition of a $(p,q)$-torus knot.

If on the other hand $u_{n}(t)$ has winding number one with respect
to the $z$-axis -- as in the case of the orbit $\tilde{u}(t)$ in Figure~\ref{fig:fiveBod} 
(see the far left frame of that figure)--
then the choreography $Q_{n}(t)$ (whose normalized period is $6\pi $) 
has toroidal winding in the second component one more than the 
$q$ value of the resonance.   So even though the choreography illustrated 
in Figure~\ref{fig:fiveBod} is resonant with order $p =3$ and $q=1$, the corresponding 
choreography is a $(3,2)$-torus knot 
after taking into account the non-trivial winding about the $z$-axis.
We conclude that the choreography --illustrated in the center right and far right 
frames of Figure~\ref{fig:fiveBod} --is a $(3,2)$-torus knot: that is, a trefoil knot.
}
\end{remark}

Following exactly the same approach as in Theorem~\ref{thm:trefoil}, we prove the 
existence of several choreographies for $n=4$, $n=7$ and $n=9$ bodies. Results 
from several of our proofs are illustrated in Figures \ref{fig:fourBod_9_14},
\ref{fig:sevBod_11_15}, and \ref{fig:nineBod_13_10} for four, seven, and nine 
bodies respectively.  The computer-assisted proofs are obtained by running the 
codes {\tt proofs\_four\_bodies.m}, {\tt proofs\_seven\_bodies.m} and 
{\tt proofs\_nine\_bodies.m}. The approximations can be found in the data files 
{\tt pts\_four\_bodies.mat}, {\tt pts\_seven\_bodies.mat} and {\tt pts\_nine\_bodies.mat}. 
All files are available at \cite{webpage}. 
In Table~\ref{table:initial_conditions} of the Appendix, the initial conditions $\tu(0) = (x_0,y_0,z_0)$ 
for the $n$-th body of each proven choreography is available. 
In Table~\ref{table:proofs_data} of the Appendix, 
some data for the proofs are given. For each of these proofs, the existence of countably 
many choreographies with near frequencies follows from Corollary~\ref{cor:nonDegen} 
and the discussion thereafter.  The reader interested in reproducing the 
choreographies via numerical integration will find 
at \cite{webpage} the data files containing initial conditions -- in inertial coordinates--
for each of the $4,7$, and $9$ body choreographies illustrated in the Figures.

\begin{figure}[h!]
\centering
\includegraphics[width = 5in]{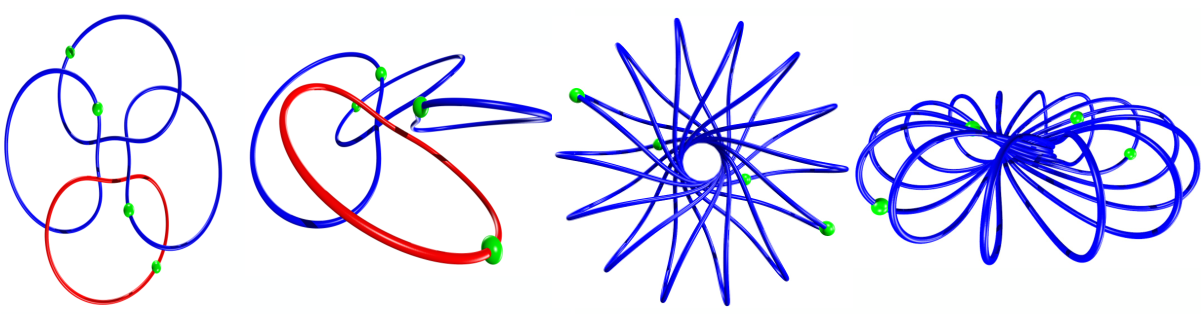}
\caption{Example result: a choreography
in the axial family for the four body 
problem ($n=4$) with $k=2$ and resonance $p:q = 14:9$.  
The bodies are shown green.  The orbit in the rotating frame 
is illustrated by the left two curves.  Far left is top down view of the orbit
projected into the $xy$ plane.  Second from left is a spatial projection, that is a 
side view of the torus.  The red loop is the segment whose existence is 
proven by studying the DDE.  The remaining three loops are obtained by symmetry. 
Since the red curve has trivial winding with respect to the $z$-axis,
the choreography is a $(14,9)$-torus knot.  In particular, 
since $p, q \neq \pm 1$ the knot is nontrivial in $\mathbb{R}^3$. 
The right two curves are the same orbit transformed back to inertial coordinates so 
that we see the torus knot choreography.  The center right frame is 
a top down view and the far right is a spatial projection of the choreography. }
\label{fig:fourBod_9_14}
\end{figure}

\begin{figure}[t!]
\centering
\includegraphics[width = 5in]{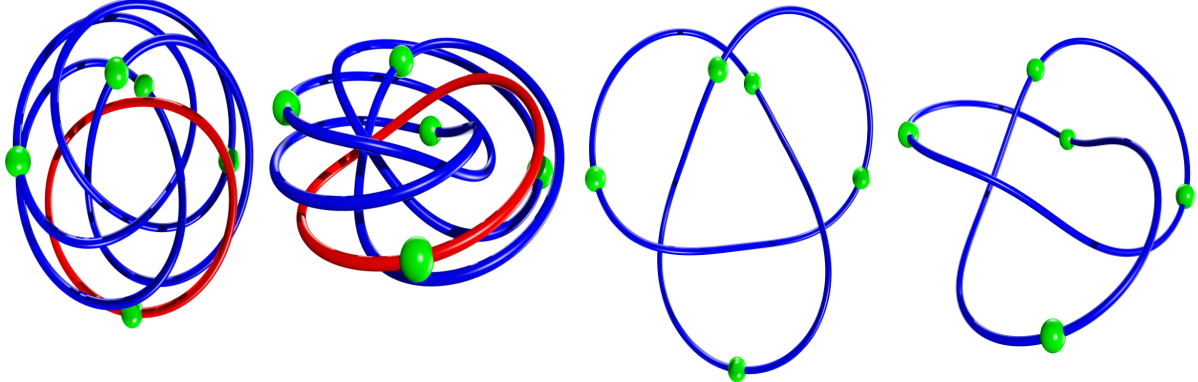}
\caption{Example result: a choreography in the axial family for the five body 
problem ($n=5$) with $k=3$ and resonance $p:q = 3:1$.  Curves from left to right have 
the same meaning as described in the caption of Figure \ref{fig:fourBod_9_14}.
Due to the nontrivial winding of the red curve around the $z$ axis, 
the choreography is a (nontrivial)
$(3,2)$-torus knot: the trefoil knot mentioned in the introduction.  See also 
Remark \ref{rem:versus}.}
 \label{fig:fiveBod}
\end{figure}

\begin{figure}[t!]
\centering
\includegraphics[width = 5.5in]{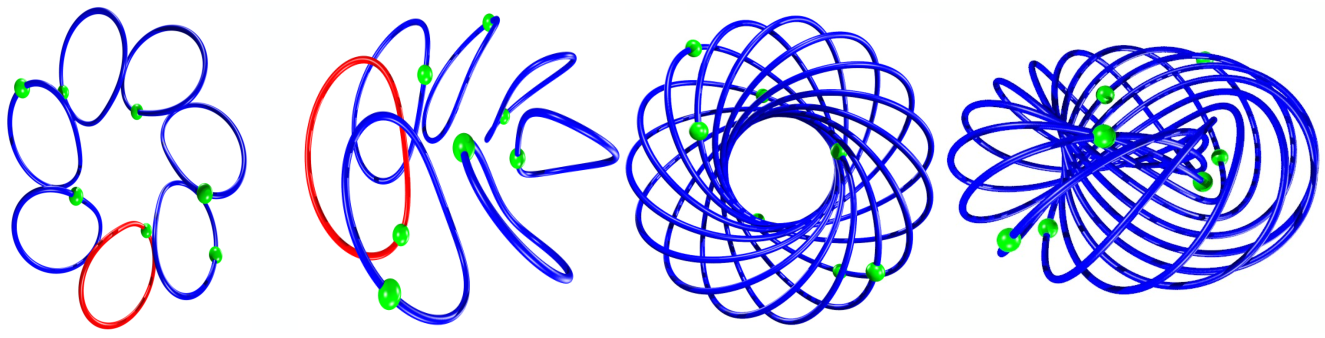}
\caption{Example result: 
a choreography in the axial family for the seven body 
problem ($n=7$) with $k=2$ and resonance $p:q=15:11$.  Curves from left to right have 
the same meaning as described in the caption of Figure \ref{fig:fourBod_9_14}.
Since the red curve has trivial winding with respect to the $z$-axis, the choreography is a 
$(15,11)$-torus knot. In particular, since $p, q \neq \pm 1$ the knot is nontrivial.}
 \label{fig:sevBod_11_15}
\end{figure}

\begin{figure}[t!]
\centering
\includegraphics[width = 5in]{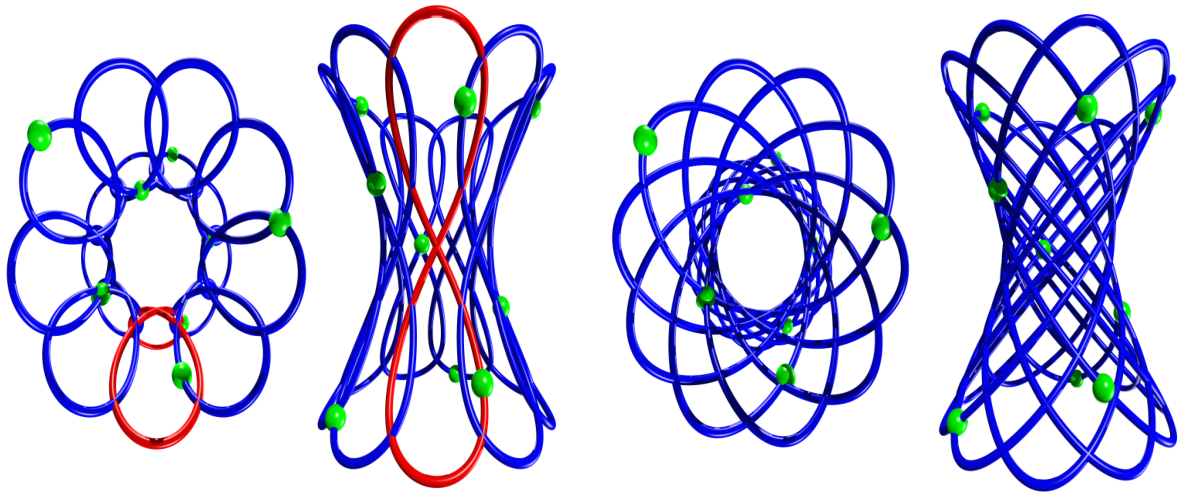}
\caption{Example result: a choreography along the Lyapunov branch for the nine body 
problem ($n=9$) with $k=7$ and resonance $p:q=10:13$.  Curves from left to right have 
the same meaning as described in the caption of Figure \ref{fig:fourBod_9_14}. In this case
the solution occurs  before the bifurcation to the axial family, hence the orbit shown here 
is \textit{not a torus knot}. Rather, the choreography 
resembles a spatial lissajous figure and illustrates the 
complexity of the vertical Lyapunov family as the number 
of bodies increases. }
 \label{fig:nineBod_13_10}
\end{figure}

\section{Acknowledgments} 
The authors wish to sincerely thank Sebius Doedel, who provided the numerical 
data for the $n$-body spatial choreographies computed in \cite{CaDoGa18}.
This data is massaged (via Fourier space Newton/continuation schemes we 
implemented in MatLab) to get appropriate approximate periodic solutions of 
the DDE, and is the starting point for all the analysis in the present manuscript.  
This material is based upon work supported by the National Science Foundation under
Grant No. 1440140, while R.C. was in residence at the Mathematical Sciences Research
Institute in Berkeley, California, during the Fall of 2018.
R.C., J.-P.L., and J.D.M.J. were partially supported by a UNAM-PAPIIT project IA102818 and IN101020. C.G.A  was partially supported by a UNAM-PAPIIT project IN115019.
J.D.M.J was partially supported by NSF grant DMS-1813501. J.-P. L. was partially supported by NSERC. 

%
%
%\begin{figure}[t!]
%\centering
%\includegraphics[width = 5.5in]{fourCurves_7_12_17_long}
%\caption{.}
% \label{fig:sevBod_12_17}
%\end{figure}

\appendix
\section*{Appendix} 
Tables \ref{table:trefoil_data}, \ref{table:initial_conditions}, and \ref{table:proofs_data}
in this appendix contain numerical data needed in the proofs
discussed in the main body of the present work.

\begin{table}[h!]
{\fontsize{2}{7}\selectfont
\makebox[\linewidth]{
\begin{tabular}{|c|c|c|c|c|c|c|c|}
\hline $\ell$ & $Re((u_1)_\ell)$ & $Im((u_1)_\ell)$ & $Re((u_2)_\ell)$ & $Im((u_2)_\ell)$ & $Re((u_3)_\ell)$ & $Im((u_3)_\ell)$ \\
\hline $0$ & $2.365605595111259$e-01 &            $0$              &     $-2.586486484802218$e-11      &                   $0$          &               $0$          &               $0$ \\
\hline $1$ & $2.730238208518935$e-01 & $-8.574371389918268$e-04 & $9.594366126621117$e-04 & $3.055023346756821$e-01 & $3.183998216275582$e-04 & $1.013843797046923$e-01 \\
\hline $2 $ & $ 2.685276027891537$e-03 & $ -1.686650070612183$e-05 & $ -1.686650070615572$e-05 & $ -2.685276027891251$e-03 & $ -1.623160460830562$e-04 & $ -2.584195753881417$e-02 \\
\hline $3$ & $ -4.758502906990690$e-03 & $  4.483372035078380$e-05 & $ -4.483372035075908$e-05 & $ -4.758502906990715$e-03 & $  6.528204100190321$e-05 & $  6.928820000785788$e-03 \\
\hline $4$ & $  1.883378890295841$e-03 & $ -2.366033392747678$e-05 & $  3.143172370973051$e-05 & $  2.501986861440681$e-03 & $ -1.766651472817732$e-05 & $ -1.406266734076041$e-03 \\
\hline $5$ & $ -5.999006965280112$e-04 & $  9.420748338183218$e-06 & $ -1.393289282442757$e-05 & $ -8.872280378635012$e-04 & $  6.028392369542407$e-18 & $ -5.065555035359391$e-18 \\
\hline $6$ & $  8.811248455572393$e-05 & $ -1.660505953485738$e-06 & $  3.309334986825996$e-06 & $  1.756053477440146$e-04 & $  2.212344038166846$e-06 & $  1.173950188081502$e-04 \\
\hline $7$ & $  8.498774137771074$e-06 & $ -1.868635687922962$e-07 & $ -1.868635687932195$e-07 & $ -8.498774137767349$e-06 & $ -1.286693928996992$e-06 & $ -5.852034835941615$e-05 \\
\hline $8$ & $ -1.218752697705919$e-05 & $  3.062649564066910$e-07 & $ -3.062649564072555$e-07 & $ -1.218752697706007$e-05 & $  4.623904740337318$e-07 & $  1.840039558482767$e-05 \\
\hline $9$ & $  5.555183124075654$e-06 & $ -1.570568558997190$e-07 & $  1.997034024758708$e-07 & $  7.063613764706126$e-06 & $ -9.561467803603867$e-08 & $ -3.381941146020149$e-06 \\
\hline$10$ & $ -1.507114521228540$e-06 & $  4.734666605843514$e-08 & $ -7.071044766269389$e-08 & $ -2.250818294936407$e-06 & $  3.032901443588402$e-19 & $  1.295911152109998$e-19 \\
\hline$11$ & $  2.335200664238406$e-07 & $ -8.070306668417385$e-09 & $  1.593213943378485$e-08 & $  4.610077899825280$e-07 & $  1.073253754609784$e-08 & $  3.105536094276917$e-07 \\
\hline$12$ & $  2.104244568410417$e-08 & $ -7.933840614338345$e-10 & $ -7.933840610780110$e-10 & $ -2.104244568485882$e-08 & $ -5.811254007216032$e-09 & $ -1.541283763438269$e-07 \\
\hline$13$ & $ -3.290905327417571$e-08 & $  1.344313219810082$e-09 & $ -1.344313219942516$e-09 & $ -3.290905327451859$e-08 & $  1.877957666662263$e-09 & $  4.597277465765067$e-08 \\
\hline$14$ & $  1.395325137847766$e-08 & $ -6.138803968764976$e-10 & $  7.771252183448973$e-10 & $  1.766373966056443$e-08 & $ -3.685236237414536$e-10 & $ -8.376391836754665$e-09 \\
\hline$15$ & $ -3.732874395549102$e-09 & $  1.759771959433038$e-10 & $ -2.610270876784223$e-10 & $ -5.536974981491839$e-09 & $  7.195923927456951$e-20 & $  7.183652045118199$e-20 \\
\hline$16$ & $  5.929064335140302$e-10 & $ -2.981756749601868$e-11 & $  5.665698400826084$e-11 & $  1.126593914301232$e-09 & $  3.809015322039982$e-11 & $  7.574023856816297$e-10 \\
\hline$17$ & $  4.804400016695753$e-11 & $ -2.567445994432811$e-12 & $ -2.567446178131416$e-12 & $ -4.804399996475025$e-11 & $ -1.932234422136176$e-11 & $ -3.615743781797638$e-10 \\
\hline$18$ & $ -7.591409192695278$e-11 & $  4.295939763834585$e-12 & $ -4.295939829539944$e-12 & $ -7.591409223386180$e-11 & $  5.943004126841669$e-12 & $  1.050195703792485$e-10 \\
\hline$19$ & $  3.114187654313435$e-11 & $ -1.860435538023476$e-12 & $  2.361958845379412$e-12 & $  3.953689255216923$e-11 & $ -1.143913421869850$e-12 & $ -1.914799688503148$e-11 \\
\hline$20$ & $ -8.409739529661349$e-12 & $  5.289131681837649$e-13 & $ -7.792542267774140$e-13 & $ -1.239017293452956$e-11 & $  1.543621735017767$e-19 & $  7.884177072944098$e-21 \\
\hline$21$ & $  1.333788472823494$e-12 & $ -8.809217097487679$e-14 & $  1.640470271719594$e-13 & $  2.483807275096231$e-12 & $  1.097557417499826$e-13 & $  1.661794417444211$e-12 \\
\hline$22$ & $  1.036015543763710$e-13 & $ -7.169318194234388$e-15 & $ -7.169447461203913$e-15 & $ -1.036016587196382$e-13 & $ -5.347398762479633$e-14 & $ -7.727301224303588$e-13 \\
\hline$23$ & $ -1.573324226886791$e-13 & $  1.138423940937104$e-14 & $ -1.138425538834881$e-14 & $ -1.573324586440305$e-13 & $  1.631244902296707$e-14 & $  2.254425161892293$e-13 \\
\hline$24$ & $  6.514968598200342$e-14 & $ -4.919794693209673$e-15 & $  6.284792512394748$e-15 & $  8.322595890650579$e-14 & $ -3.120164368146607$e-15 & $ -4.131838042098719$e-14 \\
\hline
\end{tabular}
}
}
\caption{Fourier coefficients of the {\em trefoil} choreography of Theorem~\ref{thm:trefoil}.}
\label{table:trefoil_data}
\end{table}

\begin{table}[h!]
{\fontsize{5}{8}\selectfont
\makebox[\linewidth]{
\begin{tabular}{|c|c|c|c|c|c|c|c|}
\hline
\multicolumn{7}{|c|}{\boldmath$n=4$, $k=2$\unboldmath} \\
\hline $p:q$ & $x_0$ & $y_0$ & $z_0$ & $\dot x_0$ & $\dot y_0$ & $\dot z_0$ \\
\hline $10:9$ & 1.084581210262490 & $0.269095117967146$ & $-0.400810670225760$ & $0.389692393529414$ & $-0.222026147390220$ & $ 0.422912633683090$ \\
\hline $6:5$ & $1.188423380831879$ & $ 0.396938948763056$ & $-0.389381587037265$ & $ 0.556815395497009$ & $-0.399232075676175$ & $ 0.462209587632568$
\\
\hline $14:11$ & $1.238763513470937$ & $ 0.472974975708732$ & $-0.376434682859180$ & $ 0.671427135322320$ & $-0.523882170109271$ & $ 0.485529706955392$
\\
\hline $18:13$ & $1.282136229445568$ & $ 0.569016024076380$ & $-0.350476579202572$ & $ 0.840303451206103$ & $-0.707131207512588$ & $ 0.504911385776339$
\\
\hline $10:7$ & $1.289649221019265$ & $ 0.602140964327029$ & $-0.337606815998459$ & $ 0.906273456937495$ & $-0.778064369372037$ & $ 0.505617404253052$
\\
\hline $14:9$ & $1.283423571908586$ & $ 0.686295696005838$ & $-0.287166965555756$ & $ 1.096549119253133$ & $-0.980065341494167$ & $ 0.475381865946370$ \\
\hline \hline
\multicolumn{7}{|c|}{\boldmath$n=5$, $k=3$\unboldmath} \\
\hline $p:q$ & $x_0$ & $y_0$ & $z_0$ & $\dot x_0$ & $\dot y_0$ & $\dot z_0$ \\ 
\hline $3:1$ & 
$0.781206112370790$ & $0.001836389542086$ & $0.000409996364153$ & $0.005730260732297$ & $-2.058041218487896$ & $-0.459483910447517$ \\
\hline \hline
\multicolumn{7}{|c|}{\boldmath$n=7$, $k=2$\unboldmath} \\
\hline $p:q$ & $x_0$ & $y_0$ & $z_0$ & $\dot x_0$ & $\dot y_0$ & $\dot z_0$ \\
\hline $15:11$ & $0.640762081428200$ & $0.304226148803711$ & $-0.474444652515547$ & $0.561266315985831$ & $0.527487897552293$ & $-0.391865391782611$ \\
\hline $17:12$ & $0.579026084137708$ & $0.405193913712767$ & $-0.483263936271178$ & $0.751751082063471$ & $0.635217619217003$ & $-0.389409004841267$ \\
\hline $19:13$ & $0.542163973849064$ & $0.463250571295820$ & $-0.484847294002918$ & $0.876087261468306$ & $0.693625834061019$ & $-0.375630471181662$ \\
\hline $23:15$ & $0.501902078466474$ & $0.521778491863104$ & $-0.481735430423762$ & $1.042108767392087$ & $0.739986909755284$ & $-0.348083591542181$ \\
\hline $25:16$ & $0.490096168583210$ & $0.536730950829510$ & $-0.479345448717921$ & $1.101770886821142$ & $0.747057526841343$ & $-0.337802168545392$ \\
\hline $2:1$ & $0.388010210558313$ & $0.551393376179951$ & $-0.422655405682646$ & $1.718638435158988$ & $0.663687207742979$ & $-0.293252731479080$ \\
\hline \hline
\multicolumn{7}{|c|}{\boldmath$n=9$, $k=7$\unboldmath} \\
\hline $p:q$ & $x_0$ & $y_0$ & $z_0$ & $\dot x_0$ & $\dot y_0$ & $\dot z_0$ \\
\hline $10:13$ & $0.649289870115096$ & $0.307019901740609$ & $-0.696068546706640$ & $0.621827399858452$ & $0.185756061650385$ & $-1.139929982269243$ \\
\hline $7:10$ & $0.625045716429457$ & $0.335012846089124$ & $-0.779750789678175$ & $0.591061134121929$ & $0.198381812020731$ & $-1.161246560979217$ \\
\hline
\end{tabular}
}
}
\caption{Initial conditions for the body $u_n$ used in the computer-assisted proofs of the torus knot choreographies for different resonances $p:q$ in the $n$-body problem, for $n=4,5,7,9$.}
\label{table:initial_conditions}
\end{table}

\begin{table}[h!]
{\fontsize{5}{8}\selectfont
\makebox[\linewidth]{
\begin{tabular}{|c|c|c|c|c|}
\hline
\multicolumn{5}{|c|}{\boldmath$n=4$, $k=2$\unboldmath} \\
\hline $p:q$ & $T$ & $m$ & $\nu$ & $r$ \\ 
\hline $10:9$ & $5.780190889966491$  & $30$ & $1.1$ & $2.5 \times 10^{-12}$ \\
\hline $6:5$ & $5.352028601820825$  & $30$ & $1.1$ & $1.1\times 10^{-11}$ \\
\hline $14:11$ &$5.046198396002492$ & $30$ & $1.1$ & $5.3\times 10^{-11}$ \\
\hline $18:13$ & $4.638424788244715$  & $50$ & $1.1$ & $7.1\times 10^{-11}$ \\
\hline $10:7$ & $4.495704025529494$  & $50$ & $1.1$ & $1.2\times 10^{-9}$ \\
\hline $14:9$ & $4.128707778547495$ & $60$ & $1.04$ & $8.9\times 10^{-8}$ \\
\hline \hline
\multicolumn{5}{|c|}{\boldmath$n=5$, $k=3$\unboldmath} \\
\hline $p:q$ & $T$ & $m$ & $\nu$ & $r$ \\ 
\hline $3:1$ & $1.785209272759583$ & $25$ & $1.03$ & $4.7 \times 10^{-10}$ \\
\hline \hline
\multicolumn{5}{|c|}{\boldmath$n=7$, $k=2$\unboldmath} \\
\hline $p:q$ & $T$ & $m$ & $\nu$ & $r$ \\ 
\hline $15:11$ & $3.035064895370178$ & $20$ & $1.15$ & $4.4 \times 10^{-9}$ \\
\hline $17:12$ & $2.921452840463272$ & $20$ & $1.11$ & $2.6 \times 10^{-8}$ \\
\hline $19:13$ & $2.831759112905190$ & $40$ & $1.07$ & $8.7 \times 10^{-11}$ \\
\hline $23:15$ & $2.699168385210632$ & $40$ & $1.05$ & $7.5 \times 10^{-11}$ \\
\hline $25:16$ & $2.648783908686700$ & $40$ & $1.04$ & $5.9 \times 10^{-11}$ \\
\hline $2:1$ & $2.069362428661484$ & $50$ & $1.04$ & $2.8 \times 10^{-10}$ \\
\hline  \hline 
\multicolumn{5}{|c|}{\boldmath$n=9$, $k=7$\unboldmath} \\
\hline $p:q$ & $T$ & $m$ & $\nu$ & $r$ \\ 
\hline $10:13$ & $4.479593949184486$ & $70$ & $1.05$ & $4.5 \times 10^{-8}$ \\
\hline $7:10$ & $4.922630713389546$ & $150$ & $1.04$ & $1.9 \times 10^{-9}$ \\
\hline
\end{tabular}  
}
}
\caption{Data for the proofs of the torus knot choreographies for different resonances $p:q$ and for $n=4,5,7,9$, in the $n$-body problem.}
\label{table:proofs_data}
\end{table}

\end{document}